\documentclass{amsart}
\usepackage{graphicx,color,epsf,mathrsfs}
\usepackage [cmtip,arrow]{xy}
\xyoption{all}
\usepackage{subfig}
\usepackage{float}

\usepackage{mathtools} 
\mathtoolsset{showonlyrefs,showmanualtags} 
\usepackage {pb-diagram,pb-xy}

\ifx\pdfpageheight\undefined
   \usepackage[dvips,colorlinks=true,linkcolor=blue,citecolor=blue,%
      urlcolor=purple]{hyperref}
   \usepackage[dvips]{graphicx}
   \makeatletter
   \edef\Gin@extensions{\Gin@extensions,.mps}
   \makeatother
\else
   \usepackage[bookmarksopen=false,pdftex=true,breaklinks=true,%
      backref=page,pagebackref=true,plainpages=false,%
      hyperindex=true,pdfstartview=FitH,colorlinks=true,%
      pdfpagelabels=true,colorlinks=true,linkcolor=blue,%
      citecolor=blue,urlcolor=green,hypertexnames=false%
      ]%
   {hyperref}
\fi
\usepackage{bm}

\usepackage{amsmath,amssymb}
\newtheorem{theorem}{Theorem}[section]
\newtheorem{lemma}[theorem]{Lemma}

\newtheorem{proposition}[theorem]{Proposition}

\theoremstyle{definition}
\newtheorem{definition}[theorem]{Definition}

\newtheorem{notation}[theorem]{Notation}

\theoremstyle{remark}
\newtheorem{remark}[theorem]{Remark}

\definecolor{DarkBlue}{rgb}{0,0.1,0.55}

\numberwithin{equation}{section}

\newcommand {\hide}[1]{}

\newcommand {\junk}[1]{}

\newcommand {\la}   {{\langle}}
\newcommand {\ra}   {{\rangle}}

\newcommand {\dist} {{\rm dist}}

\newcommand {\Cr} {{\rm Cr}}

\newcommand {\Def} {{\rm Def}}

\def\addots{\mathinner{\mkern1mu
\raise1pt\vbox{\kern7pt\hbox{.}}
\mkern2mu\raise4pt\hbox{.}\mkern2mu
\raise7pt\hbox{.}\mkern1mu}}

\newcommand{\blist}
    {\begin{list}
    {\textbf{\arabic{enumi}.}}
    {\setlength{\leftmargin}{0.365 in}
    \setlength{\rightmargin}{0.365 in}
    \usecounter{enumi}}}
\newcommand{\alist}
    {\begin{list}
    {\textbf{(\alph{enumii})}}
    {\setlength{\leftmargin}{0.365 in}
    \setlength{\rightmargin}{0.365 in}
    \usecounter{enumii}}}
\newcommand{\rlist}
    {\begin{list}
    {\textbf{(\roman{enumiii})}}
    {\setlength{\leftmargin}{0.365 in}
    \setlength{\rightmargin}{0.365 in}
    \usecounter{enumiii}}}
\newcommand{\Rlist}
    {\begin{list}
    {\textbf{(\Roman{enumii})}}
    {\setlength{\leftmargin}{0.365 in}
    \setlength{\rightmargin}{0.365 in}
    \usecounter{enumii}}}
\newcommand{\elist}{\end{list}}

\newcommand{\be}{\begin{equation}}
\newcommand{\ee}{\end{equation}}

\title{Hausdorff approximations and volume of tubes of singular algebraic sets}
\author{Saugata Basu, Antonio Lerario}
\thanks{Basu was supported in part by the NSF grant CCF-1910441.}

\begin{document}

\maketitle
\begin{abstract}We prove bounds for the volume of neighborhoods of algebraic sets, in the euclidean space or the sphere, in terms of the degree of the defining polynomials, the number of variables and the dimension of the algebraic set, without any smoothness assumption. This generalizes previous work of Lotz \cite{Lotz} on smooth complete intersections in the euclidean space and of B\"urgisser,  Cucker  and Lotz \cite{BCL} on hypersurfaces in the sphere, and gives a complete solution to \cite[Problem 17]{BuCu}.
\end{abstract}

\section{Introduction}
In this paper we deal with the following basic problem: given a real algebraic set $Z$ of dimension $m$, defined in $\mathbb{R}^n$ or in $S^n$ by polynomials of degree bounded by $\delta$, estimate the volume of the set of points in the ambient space which are inside a ball of radius $\sigma>0$ and are at distance at most $\varepsilon>0$ from $Z$. 

The study of the volume of neighborhoods of algebraic sets has a long history, see for instance \cite{BCSS,BCL, Demmel, Edelman, Gray1, Gray2, Loeser, Lotz,  Wongkew}, and it has fundamental algorithmic implications, e.g. for estimating the size of ill--conditioned inputs in numerical analysis (see the monograph \cite{BuCu}). In fact, one of our main motivations for this study is to give a positive answer to \cite[Problem 17]{BuCu}, see Section \ref{sec:ill} below.

The problem stated above is studied  in \cite{Lotz} in the case $Z$ is a smooth complete intersection in $\mathbb{R}^n$, and in \cite{BCL} in the case $Z$ is a hypersurface (possibly singular) in the sphere $S^n$. Here we will prove bounds with no smoothness assumption and no restriction on the dimension of $Z$.  More precisely, our first result is the following theorem, which deals with the case the ambient space is $\mathbb{R}^n$, and generalizes \cite[Theorem 1.1]{Lotz} to the singular case (see Theorem \ref{thm:affine} below for a more detailed statement). In this context it is natural to state the result in probabilistic terms.

\begin{theorem}\label{thm:affineintro}Let $\mathcal{
F}  \subset \mathbb{R}[X_1,\ldots,X_n]$ be a finite set of polynomials with degrees bounded by $\delta$ and $Z\subset\mathbb{R}^n$ be their common zero set. Assume $\dim_\mathbb{R} (Z)\leq m$.  Given $p\in \mathbb{R}^n$ and $\sigma>0$ let $x\in B(p, \sigma)$ be a uniformly distributed point\footnote{Here we turn the ball $B(p, \sigma)$ into a probability space using the Lebesgue measure normalized by the volume of the ball itself.}. Then, for every $\varepsilon>0$
\be\label{eq:affine2intro}\mathbb{P}\left(\mathrm{dist}(x, Z)\leq \varepsilon\right)\leq 4 \left(\frac{4n\delta\varepsilon}{\sigma}\right)^{n-m}\left(1+\frac{(4\delta+1)\varepsilon}{\sigma}\right)^m,\ee
and, if  $\varepsilon\leq \frac{\sigma}{(4\delta+1)m}$, 
\be\label{eq:affine3intro} \mathbb{P}\left(\mathrm{dist}(x, Z)\leq \varepsilon\right)\leq 4e \left(\frac{4n\delta\varepsilon}{\sigma}\right)^{n-m}.\ee
\end{theorem}
As one can see from \eqref{eq:affine3intro}, the codimension $c\geq n-m$ of the algebraic set $Z$ plays a key role in these estimates: it is crucial (especially for algorithms) that the dependence of the bound, for fixed $c$, is \emph{polynomial} in $n$ (the reader should think that $c$ is fixed, $n$ and $\delta$ are large, $\varepsilon>0$ is small and $\sigma>0$ is of order $O(1)$). 

The fact that a quantitative estimate of this type, as a function of the various ingredients, is possible at all follows from Weyl's Tube Formula \cite{Weyl}, which is the main tool used in the smooth case. It is intuitively clear that, as $\varepsilon\to0,$ the desired bound should be of the order $O(\varepsilon^{n-m})$, but an inductive limit argument using \cite{Weyl} on the singular points stratification of $Z$ is delicate, and the bounds depend on the complexity of the stratification.

Instead here we propose a different approach: to approximate the set $Z$ with a family of algebraic sets $\{Z_t\}_{t>0}$ which converges in the Hausdorff metric to $Z$, and such that for all $t>0$ the set $Z_t$ is a smooth complete intersection of dimension $m$, defined by polynomials of degree bounded by $2\delta$. This last condition ensures that one can apply Weyl's Tube Formula to each $Z_t$, and produce a good quantitative bound for the volume of tubes; the Hausdorff convergence $Z_t\to Z$ allows to infer bounds on the volume of tubes also as $t\to 0$. Notice again the subtlety on the role of the codimension of $Z$: every real algebraic set $Z=Z(\{P_1, \ldots, P_a\},\mathbb{R}^n)$\footnote{For the rest of the paper, given a family $\mathcal{P}$ of polynomials, we denote by $Z(\mathcal P, X)$ their common zero set, where $X$ will be $\mathbb{R}^n$ or $S^n$.} can be defined by a single polynomial $Q:=\sum_{i} P_i^2$, and the sets $\{Q=t\}$ for $t>0$ small enough are smooth complete intersections converging to $Z$ inside any ball, but they are all \emph{hypersurfaces} (i.e. they don't have the same dimension of $Z$, unless $Z$ is also a hypersurface). Our construction of the family $\{Z_t\}_{t>0}$ is more refined, and involves instead polar varieties, following \cite{Barone-Basu}. We present this idea in Theorem \ref{thm:approx} below, which is our main technical result, and which may be of independent interest.

\begin{remark}There is an alternative approach to the above problem, using the theory of multidimensional variations, introduced by Vitushkin \cite{Vit1, Vit2} and developed by Comte and Yomdin \cite{ComteYomdin}. Using this approach we can get the following bound (see Remark \ref{rem:comteyomdin} below):
\be \mathbb{P}\left(\mathrm{dist}(x, Z)\leq \varepsilon\right)\leq n\pi^{\frac{n-1}{2}}2^{n+\frac{n}{2}}n!(n+1)!^{\frac{1}{2}}\Gamma\left(\frac{n-m}{2}\right)\left(\frac{2\delta\varepsilon}{\sigma}\right)^{n-m}\left(1+\frac{(4\delta+1)\varepsilon}{\sigma}\right)^m, \ee
which has the same ``shape'' as \eqref{eq:affine2intro}, but has a dependence in $n$ which is exponential (this should be no surprise, given the greater generality of \cite{ComteYomdin}, which deals with definable sets.).
It is not clear if our technique can be extended to the definable setting, the main obstacle being the extension of the definition of polar varieties and their properties coming from complex algebraic geometry.
\end{remark}

In the case the ambient space is the sphere, we prove the following theorem, which generalizes \cite[Theorem 21.1]{BuCu} and makes it sensitive to the codimension of $Z$ (again, see Theorem \ref{thm:gS} for a more detailed statement). 

\begin{theorem}\label{thm:gSintro}Let $\mathcal{P}  \subset \mathbb{R}[X_0,\ldots,X_n]$ be a finite set of homogeneous polynomials of degree bounded by $\delta$ and $Z\subset S^n$ be their common zero set. Assume $\dim_\mathbb{R} (Z)\leq m.$  Given $p\in S^n$ and $\sigma>0$ let $x\in B(p, \sigma)$ be a uniformly distributed point. Then, for every $\varepsilon>0$
\be\label{eq:boundintro} \mathbb{P}(\mathrm{dist}(x, Z)\leq \varepsilon)\leq  2e\left(1+\frac{8\pi^3}{15}\right)\left(\frac{8n \delta \sin \varepsilon}{\sin \sigma}\right)^{n-m}\left(1+(8n\delta +8\delta+1)\frac{\sin \varepsilon}{\sin \sigma}\right)^{m}.
\ee
In particular, if $\sin \varepsilon\leq\frac{\sin \sigma}{(8n\delta +8\delta+1)m}$,
\be \mathbb{P}(\mathrm{dist}(x, Z)\leq \varepsilon)\leq  2e\left(1+\frac{8\pi^3}{15}\right)\left(\frac{8n \delta \sin \varepsilon}{\sin \sigma}\right)^{n-m}.\ee
\end{theorem}

We observe that the previous bound \eqref{eq:boundintro} has a shape which is similar to \cite[Theorem 1.3]{BCL2}, where the case of a complex algebraic subset of $\mathbb{C}\mathrm{P}^n$ is discussed.
The strategy for the proof is the same as for Theorem \ref{thm:affineintro}: we use Theorem \ref{thm:approx} to approximate $Z$ by complete intersections $\{Z_{t}\}_{t>0}$ of the same dimension as $Z$ and with degree bounded by $2\delta$, then we apply an estimate for the case of complete intersections and pass this estimate to the limit as $t\to 0$. Compared with the affine case (where we could use the bound for $Z_t$ proved by Lotz in \cite{Lotz}) there is an extra step in the spherical case: here we also need to produce the bound for the case of nonsingular complete intersections. While the strategy of proof is similar to \cite{Lotz, BCL}, via integral geometry, there are some needed modifications.  We deal with this in Section \ref{sec:preliminaries}. 

\subsection{Condition Numbers of Real Problems with High Codimension of Ill--Posedness}\label{sec:ill}
In this section we show how to interpret the previous result to give a solution to \cite[Problem 17]{BuCu}.
Recall first the following \cite[Definition 2.32]{BuCu}.
\begin{definition}Let $a, b\in S^n$. We define:
\be \mathrm{d}_{\sin}(a, b):=\sin \theta\in [0,1],\ee
where $\theta \in [0, \pi]$ is the angle between $a$ and $b$.
\end{definition}
If now $\Sigma\subset S^n$ is a symmetric cone (i.e. $\Sigma=-\Sigma$), following \cite[Chapter 21]{BuCu} one can define the conic condition number $\mathscr{C}:S^n\to \mathbb{R}$ by
\be \mathscr{C}(a):=\frac{1}{\mathrm{d}_{\sin}(a, \Sigma)}.\ee
In this context, for $u\in [0,1]$, we denote by
\be \label{eq:conic1}B_{\sin}(a, u)=\{\mathrm{d}_{\sin}(a, \cdot)\leq u\}=B_{S^n}(a, \arcsin u)\cup B_{S^n}(-a, \arcsin u).\ee
Next theorem is a generalization of \cite[Theorem 21.1]{BuCu}, which corresponds to the case $m=n-1$ (the proof is given in Section \ref{sec:pill}).

\begin{theorem}\label{thm:ill}Let $\mathscr{C}$ be a conic condition number with set of ill--posed inputs $\Sigma$. Assume that $\Sigma$ is contained in an algebraic set $Z\subset S^n$ defined by homogeneous polynomials of degree bounded by $\delta$ and of dimension $\dim_{\mathbb{R}}(Z)\leq m$. Then for all $0<u\leq 1$ and for all $t\geq\frac{m(8n \delta+8\delta+1)}{u}$:
\be \label{eq:coninc2}\sup_{a\in S^n}\underset{x\in B_{\sin}(a, u)}{\mathbb{P}}\left\{\mathscr{C}(x)\geq t\right\}\leq 2e\left(1+\frac{8\pi^3}{15}\right)\left(\frac{8n\delta }{ut}\right)^{n-m}.
\ee
In particular (take $u=1$), for all $t\geq m(8n \delta+8\delta+1)$,
\be {\mathbb{P}}\left\{\mathscr{C}(x)\geq t\right\}\leq 2e\left(1+\frac{8\pi^3}{15}\right) \left(\frac{8n\delta}{ t}\right)^{n-m}. \ee
\end{theorem}
\subsection{Structure of the paper}
The rest of the paper is organized as follows. In Section~\ref{sec:hausdorff}, 
we prove some basic results on Hausdorff limits of semialgebraic subsets of $\mathbb{R}^n$. In particular, in Proposition \ref{propo:HL} we give a description of the Hausdorff limit of a one--parameter
semialgebraic family of bounded semi-algebraic subsets of euclidean space. We use this in Proposition \ref{propo:zeta} and relate it to limits of bounded semialgebraic sets defined over non-Archimedean extensions of $\mathbb{R}$, in order to utilize certain results proved in \cite{Barone-Basu}. 
These results are then used to prove an approximation result (cf. Theorem~\ref{thm:approx}) which is a key technical result of the paper. In Section~\ref{sec:affine}, we prove Theorem~\ref{thm:affine} after introducing some preliminary results, including a bound proved by Lotz in the non-singular case (cf. Theorem~\ref{thm:Lotz}). In Section~\ref{sec:spherical}, we treat the spherical case. We first prove an analog of Theorem~\ref{thm:Lotz} in the spherical case (cf. Theorem~\ref{thm:completeS}). We then prove Theorems~\ref{thm:gSintro} and \ref{thm:ill}.

\section{Hausdorff approximations}
\label{sec:hausdorff}
\subsection{Metric geometry}
\begin{notation}We will mostly be dealing with three metric spaces:
\begin{enumerate}
\item The euclidean space $\mathbb{R}^n$ with the standard metric: $\mathrm{dist}_{\mathbb{R}^n}(a,b)=\|a-b\|\quad \forall a,b\in \mathbb{R}^n$.
\item The sphere $S^n\hookrightarrow \mathbb{R}^{n+1}$, with the riemannian metric induced by the ambient space. The distance between two points $a, b\in S^n$ equals the length of the shortest geodesic on the sphere joining them: $\mathrm{dist}_{S^n}(a,b)=\arccos \langle a, b\rangle.$ The diameter of the sphere for this metric is $\pi$.
\item Since the antipodal map $x\mapsto -x$ is an isometry of the sphere, the riemannian metric on the sphere descends to a riemannian metric on $\mathbb{R}\mathrm{P}^n$.  The distance $\mathrm{dist}_{\mathbb{R}\mathrm{P}^n}([a],[b])$ between two points $[a], [b]\in \mathbb{R}\mathrm{P}^n$ equals the length of the shortest geodesic on the the projective space joining them.  The projective space is locally isometric to the sphere, but its diameter is $\frac{\pi}{2}$. \end{enumerate}
When the metric space $X$ is clear from the context, we denote simply by $\mathrm{dist}(x,y)$ the distance between two points $x,y\in X$ and, for $r\geq 0$, by $B(x, r)$ the closed ball of radius $r$ around $x\in X$. In the above cases the metric comes from a riemannian structure on $X$. The riemannian structure induces a volume density $\omega_X$ and we denote by $``\omega (\mathrm{dx})$'' integration with respect to this density. For a Borel set $A\subseteq X$ we denote its volume by $\mathrm{vol}(A):=\int_A \omega_{X}(\mathrm{d}x).$
\end{notation}

\begin{definition}Let $X$ be a metric space and $C\subseteq X$ be a closed subset. For $\varepsilon\geq 0$ we denote by $\mathcal{U}_X(C, \varepsilon)$ the $\varepsilon$--neighborhood of $C$ in $X$:
\be \mathcal{U}_X(C, \varepsilon):=\{x\in X\,|\, \mathrm{\dist}(x, C)\leq \varepsilon\}=\bigcup_{x\in C}B(x, \varepsilon).\ee
(We will omit the subscript and simply write $\mathcal{U}(C, \varepsilon)$ when the ambient space $X$ is clear from the context.)
If $\{C_t\}_{t>0}$ is a family of closed sets in $X$, we will write ``$\lim_{t\to0}C_t=C_0$'' if there is a closed set $C_0\subseteq X$ such that for every $\varepsilon>0$ there exists $t_\varepsilon>0$ such that for all $0<t<t_\varepsilon$
\be C_t\subseteq \mathcal{U}_X(C_0, \varepsilon)\quad \textrm{and}\quad C_0\subseteq \mathcal{U}_X(C_t, \varepsilon).\ee
This means that the family $\{C_t\}_{t>0}$ converges to $C_0$ in the Hausdorff metric. The notation ``$\lim_{t\to0}C_t\supseteq C$'' means that the family $\{C_t\}_{t>0}$ converges to some closed set $C_0\subseteq X$ and that $C_0\supseteq C$ (analogously for the notation ``$\lim_{t\to0}C_t\subseteq C$'').

\end{definition}
\begin{theorem}\label{thm:metric}Let $X$ be a metric space. Let $C\subseteq X$ be a closed set and $\{C_t\}_{t>0}$ be a family of closed sets such that there exists a closed set $B\subseteq X$ with the property that:
\be \lim_{t\to 0}\left(C_t\cap B\right)\supseteq C\cap B.\ee
Then, for every $\tau>0$ there exists $t_\tau>0$ such that for all $0<t<t_\tau$ and for all $p\in X$ and $\sigma, \varepsilon>0$ such that $B(p, \sigma+\varepsilon)\subseteq B,$ we have:
\be\label{eq:0} \mathcal{U}_X(C, \varepsilon)\cap B(p, \sigma)\subseteq \mathcal{U}_X(C_t, \varepsilon+\tau)\cap B(p, \sigma).\ee
\end{theorem}
\begin{proof}
By assumption, the Hausdorff limit inside $B$ of the family $\{C_t\cap B\}_{t>0}$ contains $B\cap C$ and therefore, given $\tau>0$ there exists $t_\tau>0$ such that for all $0<t<t_\tau$:
\be\label{eq:1} C\cap B\subseteq \lim_{t\to 0} \left(C_t\cap B\right)\subseteq \mathcal{U}_{B}(C_t\cap B, \tau)\subseteq \mathcal{U}_X(C_t\cap B, \tau).\ee

Observe now that for every $\sigma, \varepsilon>0$ and $x\in X$ such that $B(p, \sigma+\varepsilon)\subseteq B$, we have the following inclusion:
\be\label{eq:2} \mathcal{U}_X(C, \varepsilon)\cap B(p, \sigma)\subseteq \mathcal{U}_X(C\cap B, \varepsilon).\ee
In order to prove this inclusion, we notice that for every point $x\in \mathcal{U}_{X}(C, \varepsilon)\cap B(p, \sigma)$ there exists $z\in C$ such that $\mathrm{dist}(x,z)\leq \varepsilon.$ Since $x\in B(p, \sigma)$ then, by triangle inequality, $\mathrm{dist}(p, z)\leq\mathrm{dist}(p, x)+\mathrm{dist}(x,z)\leq \sigma+\varepsilon$ and $z\in B(p, \sigma+\varepsilon)\subseteq B.$ Therefore for every $x\in \mathcal{U}_X(C, \varepsilon)\cap B(p, \sigma)$ there exists $z\in C\cap B$ such that $\mathrm{dist}(x, z)\leq \varepsilon$ and $x\in  \mathcal{U}_X(C\cap B, \varepsilon).$

Now, given $\tau,\varepsilon>0$ and $t<t_\tau$, we also have the inclusion:
\be\label{eq:3}  \mathcal{U}_{X}(C\cap B, \varepsilon)\subseteq \mathcal{U}_X(C_t, \varepsilon+\tau).\ee
In fact, by \eqref{eq:1}, for every $\varepsilon>0$ we have:
\be \mathcal{U}_X(C\cap B, \varepsilon)\underset{\eqref{eq:1}}{\subseteq}\mathcal{U}_X(\mathcal{U}_X(C_t\cap B, \tau), \varepsilon)\subseteq \mathcal{U}_X(C_t\cap B, \varepsilon+ \tau)\subseteq \mathcal{U}_X(C_t, \varepsilon+ \tau),\ee
and \eqref{eq:3} follows.

Therefore, for every $\tau>0$ there exists $t_\tau>0$ such that for every $\sigma, \varepsilon>0$ and $x\in X$ such that $B(p, \sigma+\varepsilon)\subseteq B$ and for all  $0<t<t_\tau$ we have:
\be\label{eq:4} \mathcal{U}(C, \varepsilon)\cap B(p, \sigma)\subseteq \mathcal{U}(C_t, \varepsilon+\tau).\ee

Intersecting both sides of \eqref{eq:4} with $B(p, \sigma)$ gives \eqref{eq:0}.
\end{proof}

The following lemma is elementary, but it will be useful in the sequel.
\begin{lemma}\label{lemma:covering}Let $\{C_t\}_{t>0}$ be a family of closed sets in $\mathbb{R}\mathrm{P}^n$ converging to some closed set $C_0:=\lim_{t\to 0}C_t$. Denoting by $q:S^n\to \mathbb{R}\mathrm{P}^n$ the covering map, we have:
\be \lim_{t\to 0}q^{-1}(C_t)=q^{-1}(C_0).\ee
\end{lemma}
\begin{proof}We observe first that for every closed set $Y\subseteq \mathbb{R}\mathrm{P}^n$ and for $\varepsilon<\frac{\pi}{4}$ we have:
\be \label{eq:tubeq}q^{-1}\left(\mathcal{U}_{\mathbb{R}\mathrm{P}^n}(Y, \varepsilon)\right)=\mathcal{U}_{S^n}(q^{-1}(Y), \varepsilon).\ee
Let us write the condition that $C_0=\lim_{t\to 0}C_t$: for every $\varepsilon>0$ there exists $t_\varepsilon>0$ such that for all $0<t<t_\varepsilon$:
\be \label{eq:L}C_t\subseteq \mathcal{U}_{\mathbb{R}\mathrm{P}^n}(C_0, \varepsilon)\quad \textrm{and}\quad C_0\subseteq\mathcal{U}_{\mathbb{R}\mathrm{P}^n}(C_t, \varepsilon).\ee
Applying $q^{-1}(\cdot)$ to both the inclusions in \eqref{eq:L}, and using \eqref{eq:tubeq}, gives precisely the condition for the convergence $\lim_{t\to 0}q^{-1}(C_t)=q^{-1}(C_0)$.
\end{proof}

\subsection{Hausdorff limits of semialgebraic sets}
In this section we give a simple description of Hausdorff limits in the semialgebraic world, and related it to the notion of limits of bounded semialgebraic sets defined over non--Archimedean extensions of $\mathbb{R}$
\begin{proposition}\label{propo:HL}Let $B\subset \mathbb{R}^n$ be a bounded semialgebraic set and $A\subseteq B\times (0, \infty)$ be a semialgebraic set. Denoting by $p_1:B\times [0, \infty)\to B$ the projection on the first factor and by $p_2:B\times [0, \infty)\to [0, \infty)$ the projection on the second factor, define for every $t>0$ the set 
\be A_t:=p_1(p_2^{-1}(t)\cap A).\ee
Let $\mathrm{clos}(A)\subseteq \mathrm{clos}(B)\times [0, \infty)$ be the closure of $A$ and set $A_0:=p_1(p_2^{-1}(t)\cap \mathrm{clos}(A)).$ Then
\be \lim_{t\to 0}A_t=A_0.\ee
\begin{proof}We need to prove that for every $\varepsilon>0$ there exists $t_\varepsilon>0$ such that for all $0<t<t_\varepsilon$ we have:
\be \label{eq:and}A_t\subseteq \mathcal{U}(A_0, \varepsilon)\quad \textrm{and}\quad A_0\subseteq \mathcal{U}(A_t, \varepsilon).\ee
We prove the two inclusions \eqref{eq:and} separately, arguing by contradiction. 

Assume first that there exists $\varepsilon>0$ such that for every $n>0$ there exist $0<t_n\leq\frac{1}{n}$ and $a_{t_n}\in A_{t_n}$ such that for every $a_0\in A_0$
\be \label{eq:e1}\mathrm{dist}(a_0, a_{t_n})\geq \varepsilon.\ee
Then, up to subsequences, since $t_n\to 0$ and $B$ is bounded, we can assume that $(a_{t_n}, t_n)\to (a_0, 0)\in \mathrm{clos}(A).$ This means $a_{t_n}\to a_0\in A_0$, which contradicts \eqref{eq:e1} and proves the first inclusion in \eqref{eq:and} (notice that we did not use the semialgebraic hypothesis for this inclusion).

As for the other inclusion, assume again by contradiction that there exists $\varepsilon>0$ such that for every $n>0$ there exists $0<t_n\leq\frac{1}{n}$ and $a_{0,n}\in A_0$ such that for all $a_{t_n}\in A_{t_n}$ we have $\mathrm{dist}(a_{0,n}, a_{t_n})\geq \varepsilon.$ Up to subsequences, we can assume $a_{0,n}\to a_0\in A_0$, and therefore for every $n>0$ there is $0<t_n\leq\frac{1}{n}$ such that for every $a_{t_n}\in A_{t_n}$
\be \label{eq:2A} \mathrm{dist}(a_0, a_{t_n})\geq \varepsilon.\ee
Now we use the semialgebraic hypothesis: by the Curve selection Lemma \cite[Theorem 3.19]{BPRbook2} there exists a semialgebraic arc $\gamma:[0, \delta)\to \mathrm{clos}(A)\subseteq \mathrm{clos}(B)\times [0, \infty)$ such that $\gamma(s)\in A$ for all $s\in (0, \delta)$ and $\gamma(0)=a_0.$ Let us write $\gamma(s)=(a(s), t(s))$, then the function $t(s)$ is also semialgebraic and we may assume that it is injective for $0<s<\delta',$ for some $0<\delta'\leq \delta.$ In particular for every $n>0$ there is $s_n\in (0, \delta')$ such that $a_{t_n}:= a(s_n)\in A_{t_n}$ and $a(s_n)\to a_0$, which contradicts \eqref{eq:2A}\footnote{The Curve Selection Lemma is used to construct a sequence $a_{t_n}\in A_{t_n}$ converging to $a_0$. Without this we would only be able to construct a sequence $a_{t'_n}$ still converging to $a_0$ but with $t'_n\leq t_n$ (remember that in the quantifiers we have ``...for every $n>0$ there exists $t_n>0$...'') For instance, if one takes as $A$ the graph of the function $t\mapsto \sin\left(\frac{1}{t}\right)$, we see that the first inclusion is still true, what fails is the second one.}.
\end{proof}
\end{proposition}
\begin{notation}\label{notation:pui}Denote by $\mathbb{R}\la\zeta\ra$ the field of
algebraic Puiseux series with coefficients in $\mathbb{R}$, which coincide
with the
germs of semi-algebraic continuous functions (see \cite[Chapter 2, Section 6 and Chapter 3, Section 3]{BPRbook2}).  An element $x\in \mathbb{R}\la \zeta\ra$ is bounded over $\mathbb{R}$ if $\vert x \vert \le R$ for some $0\le
R \in \mathbb{R}$.  The subring $\mathbb{R}\la\zeta\ra_b$ of elements of $\mathbb{R}\la\zeta\ra$
bounded over $\mathbb{R}$ consists of the Puiseux series with non-negative
exponents.  We denote by
\be \lambda_{\zeta}:\mathbb{R}\la\zeta\ra_b^n\to \mathbb{R}^n\ee
the ring homomorphism which maps $\sum_{i
  \in \mathbb{N}} a_i \zeta^{i / q}$ to $a_0$. 

Let now $S\subset \mathbb{R}\langle \zeta\rangle^n$ be a semialgebraic set defined by a formula $\phi$ with constants in $\mathbb{R}[\zeta]$ such that for some $0\leq R\in \mathbb{R}$ we have $|s|\leq R$ for all $s\in S$. Then, for every $t\in \mathbb{R}$, we denote by $\Lambda_t(\phi)\subset \mathbb{R}^n$ the semialgebraic set defined by the formula obtained from $\phi$ replacing $\zeta $ by $t$. Note that for $t>0$ small enough $\Lambda_t(\phi)\subseteq B(0, R).$
\end{notation}
\begin{proposition}\label{propo:zeta}Let $\phi$ be a first order formula with constants in $\mathbb{R}[\zeta]$ such that for some $R\in \mathbb{R}, R>0$ we have $|s|\leq R$ for all $s$ such that $\phi(s)$.  Then:
\be \lambda_\zeta\left(\left\{x\in \mathbb{R}\langle \zeta\rangle^n\,|\, \phi(x)\right\}\right)=\lim_{t\to 0}\Lambda_t(\phi).\ee
\begin{proof}
Let $S = \left\{x\in \mathbb{R}\langle \zeta\rangle^n\,|\, \phi(x)\right\}$.
It is proved in \cite[proof of Proposition 12.43]{BPRbook2} that
\[
\lambda_\zeta(S) = 
\mathrm{clos}\left(\{(x,t) \in \mathbb{R}^{k+1}| x \in \Lambda_t(S),\, t > 0\}\right) \cap \left(\mathbb{R}^k \times \{0\}\right).
\]
The proposition now follows from Proposition \ref{propo:HL}. 
\end{proof}
\end{proposition}\subsection{Hausdorff approximation by complete intersections}
\begin{notation}Let $X_1, \ldots, X_n$ be linear coordinates on $\mathbb{R}^n$. Given $G\in \mathbb{R}[X_1, \ldots, X_n]$ and $0\leq k\leq n$, we denote by
$\Cr_k(G)$ the set of polynomials
\[
\Cr_k(G):=\left\{
G, \frac{\partial G}{\partial X_1},\ldots, \frac{\partial G}{\partial X_k}
\right\}\subset \mathbb{R}[X_1, \ldots, X_n].
\]
We will denote by $\Cr_k^h(G)$ the corresponding set
\[
\Cr_k^h(G):=\left\{
G^h, \frac{\partial G^h}{\partial X_1},\ldots,
\frac{\partial G^h}{\partial X_p}
\right\}\subset \mathbb{R}[X_0, \ldots, X_n]
\]
of homogenized polynomials. 
\end{notation}
The following proposition is proved in \cite{Barone-Basu}.

\begin{proposition} 
\label{prop:criticallocusontube}Let  $Q, G\in \mathbb{R}[X_1, \ldots, X_n]$ be polynomials of even degree, such that $\deg(G)\geq \deg (Q)$ and $G$ is non--negative. Define the polynomial:
\be D(Q, G, \zeta):=(1-\zeta)Q-\zeta G\in \mathbb{R}\langle \zeta\rangle[X_1, \ldots , X_n].\ee
For a generic choice of the linear coordinates\footnote{These coordinates are called ``good for $Q$'' in \cite{Barone-Basu}.} $X_1, \ldots, X_n$ on $\mathbb{R}^n$ the following is true. For every $0\leq k\leq n$ and $0<R\in \mathbb{R}$:
\be \lambda_\zeta\left(Z(\Cr_{k}(\Def(Q,G,\zeta)),\mathbb{R}\langle \zeta\rangle^n)\right)=Z(Q, \mathbb{R}^n)\cap B(0, R).\ee
\end{proposition}
\begin{proof}This is proved in \cite[Proposition 3.4]{Barone-Basu}.
\end{proof}

We are now ready to prove the main technical result of this section.

\begin{theorem}[Hausdorff approximation by complete intersections]\label{thm:approx}Let $\mathcal{P}\subset \mathbb{R}[X_0, \ldots, X_n]$ be a finite set of homogeneous polynomials such that $\max_{P\in \mathcal{P}}\deg (P)\leq d$.
\begin{enumerate}
\item Let $y\in \mathbb{R}\mathrm{P}^n $ and denote by $H\simeq \mathbb{R}\mathrm{P}^{n-1}$ the hyperplane $H:=y^{\perp}.$ Let also $V:=Z(\mathcal{P},\mathbb{R}\mathrm{P}^n \setminus H)\subset \mathbb{R}^n$ and assume that $\dim_{\mathbb{R}} V\leq m.$ Then there exists a one parameter family of real algebraic sets $\{V_t\}_{t>0}$ and $t_0>0$  such that for every $0<t<t_0$ the set $V_t$ is a nonsingular complete intersection in $\mathbb{R}^n$ defined by $n-m$ affine polynomials of degree bounded by $2d$, and such that for every ball $B(0,R)$ in $\mathbb{R}^n$ with $R>0$:
\be\label{eq:H1} \lim_{t\to 0}\left(V_t\cap B(0, R)\right)=V\cap B(0, R).\ee
\item Let $Z:=Z(\mathcal{P}, S^n)\subset S^n$ and assume that $\dim_\mathbb{R} (Z)\leq m$. Then there exists a one parameter family of real algebraic sets $\{Z_t\}_{t>0}$ and $t_0>0$ such that for every $0<t<t_0$ the set $Z_t$ is a nonsingular complete intersection in $S^n$ defined by $n-m$ homogeneous polynomials of degree bounded by $2d$ and such that:
\be \lim_{t\to 0}Z_t\supseteq Z.\ee
\end{enumerate}
\end{theorem}
\begin{proof}
Let us prove part (A). Up to a linear change of coordinates we can assume that $H=\{X_0=0\}$. We denote by
\be Q:=\sum_{P\in \mathcal{P}}P^2|_{X_0=1}\in \mathbb{R}[X_1, \ldots, X_n]_{\leq 2d}.\ee
Observe that $Z(Q, \mathbb{R}^n)=Z(\mathcal{P}, \mathbb{R}\mathrm{P}^n\setminus H)=V.$ Let now $G\in \mathbb{R}[X_1, \ldots, X_n]_{2d}$ such that $G\geq 0$ and such that for every $0\leq k\leq n$ the set $\Cr^h_k(G)$ defines a smooth complete intersection in $\mathbb{C}\mathrm{P}^n$ (and therefore also in $\mathbb{R}\mathrm{P}^n$). Such polynomial $G$ exists by \cite[Proposition 2.22]{Barone-Basu}. 

Up to a linear transformation in $\mathbb{R}^n$, we can assume that the coordinates are such that Proposition \ref{prop:criticallocusontube} can be applied. We define now as above the polynomial
\be D(Q, G, \zeta)=(1-\zeta)Q-\zeta G\in \mathbb{R}\langle \zeta\rangle[X_1, \ldots , X_n],\ee
and the algebraic sets:
\be \overline{V}\la\zeta\ra:=Z\left(\Cr_{n-m-1}^h(D(Q, G, \zeta)),\mathrm{P}(\mathbb{R}\langle \zeta\rangle^{n+1})\right)\ee
and 
\be V\la\zeta\ra:=Z\left(\Cr_{n-m-1}(D(Q, G, \zeta)), \mathbb{R}\langle \zeta\rangle^{n}\right), \ee
where $\mathrm{P}(\mathbb{R}\langle \zeta\rangle^{n+1})$ denotes the $n$--dimensional projective space over the field $\mathbb{R}\langle \zeta\rangle$.
For $R>0$ we introduce the following first order formulas with coefficients in $\mathbb{R}[\zeta]$:
\begin{align} \overline{\phi}&:=\left( \Cr_{n-m-1}^h(D(Q, G, \zeta))(X_0, \ldots, X_n)=0 \right)\\
\phi&:=\left( \Cr_{n-m-1}(D(Q, G, \zeta))(X_1, \ldots, X_n)=0 \right)\\
\phi_R&:=\left( \Cr_{n-m-1}(D(Q, G, \zeta))(X_1, \ldots, X_n)=0 \quad \wedge \quad X_1^2+\cdots +X_n^2\leq R^2\right).
\end{align}
Notice that $\overline{V}\la\zeta\ra=\{[y]\in \mathrm{P}(\mathbb{R}\langle \zeta\rangle^{n+1})\,|\, \overline \phi(y)\}$ and $V\la\zeta\ra=\{x\in \mathbb{R}\langle \zeta\rangle^{n}\,|\, \phi(x)\}.$ Using the convention introduced in Notation \ref{notation:pui}, we define for $t>0$:
\be \overline{V}_t:=\Lambda_t(\overline{\phi})\quad \textrm{and}\quad V_t:=\Lambda_t(\phi).\ee

Observe now that the set 
\be\Sigma=\{t \in \mathbb{C} \mid \overline{V}_t 
\mbox{ is not a non-singular complete intersection in $\mathbb{C}\mathrm{P}^n$}
\}\ee
is a Zariski closed constructible subset of $\mathbb{C}$ (using \cite[Theorem 4.102]{BPRbook2})
whose complement contains $1$. Therefore $\mathbb{C}\setminus \Sigma$ is nonempty and Zariski open, hence it is cofinite in $\mathbb{C}$. 
This implies that there exists $t_0>0$ such that for all $0<t<t_0$ the set $\overline{V}_t\subset \mathbb{R}\mathrm{P}^n$ is a nonsingular complete intersection defined by $n-m$ homogeneous polynomials of degree bounded by $2d$, and in particular so is $V_t$. Moreover, by construction we also have:
\be V_t\cap B(0, R)=\Lambda_t(\phi_R)\quad \mathrm{and}\quad V\cap B(0, R)=\lambda_\zeta\left(\{x\in \mathbb{R}\langle \zeta\rangle^{n}\,|\, \phi_R(x)\}\right),\ee
where the right--hand side identity follows from Proposition \ref{prop:criticallocusontube}. The Hausdorff limit \eqref{eq:H1} follows now from Proposition \ref{propo:zeta}.

Let us now prove part (B). We first prove a projective version of the statement. More precisely, let $\overline{V}:=Z(\mathcal{P}, \mathbb{R}\mathrm{P}^n)$ and pick a hyperplane $H\simeq \mathbb{R}\mathrm{P}^{n-1}$ such that
\be\label{eq:closure} \mathrm{clos}\left(\overline{V}\setminus H\right)=\overline{V}.\ee
The generic hyperplane $H$ has this property (to see this it is enough to take a stratification $\overline{V}=\sqcup_{j=1}^r S_j$ into smooth strata and to pick a hyperplane $H$ transversal to all strata). Observe now that the construction from part (A), applied to $V:=Z(\mathcal{P}, \mathbb{R}\mathrm{P}^n\setminus H),$ actually yields the existence of a family of projective algebraic sets $\{\overline{V}_t\}_{t>0}$ and $t_0>0$ such that for $0<t<t_0$ each $\overline{V}_t$ is a smooth complete intersection in projective space defined by $n-m$ homogeneous polynomials of degree bounded by $2d$ (and then the family $\{V_t\}_{t>0}$ for part (A) is obtained by letting $V_t=\overline{V}_t\setminus H$). We will prove that:
\be \label{eq:inc}\lim_{t\to 0}\overline{V}_t\supseteq \overline{V}.\ee
Using now part (A) of the statement, we see that for every $B(0,R)\subset \mathbb{R}^n\simeq \mathbb{R}\mathrm{P}^n\setminus H$:
\be\label{eq:limeq} \lim_{t\to 0}\left(\overline{V}_t\cap B(0,R)\right)=\lim_{t\to 0}\left(V_t\cap B(0,R)\right)=V\cap B(0, R)=\overline{V}\cap B(0,R).\ee
Let us remark that the first limit is performed in the Hausdorff metric induced by the ambient space $(B(0,R), \mathrm{dist}_{\mathbb{R}\mathrm{P}^n})$ and the second one by $(B(0,R), \mathrm{dist}_{\mathbb{R}^n})$, but the two limits are equal because the Hausdorff convergence on compact sets does not depend on the metric, but just on the topology \cite[Proposition 2.4.14]{Srivastasa}.
Equation \eqref{eq:limeq} tells that every $v\in \overline{V}\setminus H$ also belongs to $\lim_{t\to 0}\overline{V}_t$, i.e.
\be\label{eq:inc1} \lim_{t\to 0}\overline{V}_t\supseteq V.\ee
On the other hand, let $\overline v\in \overline{V}\setminus V\subset H.$ Then, because of \eqref{eq:closure}, we can apply the Curve selection Lemma \cite[Theorem 3.19]{BPRbook2} and get the existence of a semialgebraic arc $\gamma:[0, \tau)\to \mathbb{R}\mathrm{P}^n$ such that $\gamma(0, \tau)\subset V$ and $\gamma(0)=\overline{v}.$ Since $\lim_{t\to 0}\overline{V}_t$ is closed and $\gamma(s)\in\lim_{t\to 0}\overline{V}_t$ for every $s\in (0, \tau)$, then also $\overline{v}=\gamma(0)\in \lim_{t\to 0}\overline{V}_t.$ This implies that
\be \lim_{t\to 0}\overline{V}_t\supseteq \overline{V}\setminus V,\ee
which, together with \eqref{eq:inc1}, proves \eqref{eq:inc}.

Let us now go back to the spherical version. Denote by $q:S^n\to \mathbb{R}\mathrm{P}^n$ the covering map and set $Z_t:=q^{-1}(\overline{V}_t)$. We apply now Lemma \ref{lemma:covering} to the family $\{\overline{V}_t\}_{t>0}$ and, using \eqref{eq:inc}, we get:
\be \lim_{t\to 0}Z_t=\lim_{t\to 0}q^{-1}(\overline{V}_t)=q^{-1}\left(\lim_{t\to 0}\overline{V}_t\right)\supseteq q^{-1}(\overline{V})=Z.\ee
This proves part (B) of the theorem.
\end{proof}
\section{The affine case}
\label{sec:affine}
The next result is proved in \cite{Lotz} and gives a way for estimating the volume of tubes around \emph{nonsingular} complete intersections. 
\begin{theorem}[Lotz]\label{thm:Lotz}Let $V$ be the zero set  in $\mathbb{R}^n$  of polynomials $P_1, \ldots, P_c$ of degree at most $d$. Assume that $V$ is a smooth complete intersection of dimension $m=n-c$. Let $x$ be a uniformly distributed point in a ball $B(p, s)$ of radius $s$ around $p\in \mathbb{R}^n$. Then for every $r>0$
\be\label{eq:Lotz} \mathbb{P}\left(\mathrm{dist}(x, V)\leq r\right)\leq 4\sum_{i=0}^m{{n}\choose{n-m+i}}\left(\frac{2dr}{s}\right)^{n-m+i}\left(1+\frac{r}{s}\right)^{m-i}.\ee
\end{theorem}
We use this result as a tool  for proving next theorem, which deals with the case of algebraic sets in $\mathbb{R}^n$, with no regularity assumption. Notice that the result has the same shape of \eqref{eq:Lotz}, except for a doubling of the degree.

\begin{theorem}\label{thm:affine}Let $\mathcal{
F}  \subset \mathbb{R}[X_1,\ldots,X_n]$ be a finite set of polynomials with degrees bounded by $\delta$ and let $V\subset \mathbb{R}^n$ be their common zero set. Assume $\dim_\mathbb{R} (V)\leq m$.  Given $p\in \mathbb{R}^n$ and $\sigma>0$ let $x\in B(p, \sigma)$ be a uniformly distributed point. Then for every $\varepsilon>0$:
\be\label{eq:affine} \mathbb{P}\left(\mathrm{dist}(x, V)\leq \varepsilon\right)\leq 4\sum_{i=0}^m{{n}\choose{n-m+i}}\left(\frac{4\delta\varepsilon}{\sigma}\right)^{n-m+i}\left(1+\frac{\varepsilon}{\sigma}\right)^{m-i}.\ee
In particular:
\be\label{eq:affine2}\mathbb{P}\left(\mathrm{dist}(x, V)\leq \varepsilon\right)\leq 4 \left(\frac{4n\delta\varepsilon}{\sigma}\right)^{n-m}\left(1+\frac{(4\delta+1)\varepsilon}{\sigma}\right)^m,\ee
and, if  $\varepsilon\leq \frac{\sigma}{(4\delta+1)m}$, 
\be\label{eq:affine3} \mathbb{P}\left(\mathrm{dist}(x, V)\leq \varepsilon\right)\leq 4e \left(\frac{4n\delta\varepsilon}{\sigma}\right)^{n-m}.\ee
\end{theorem}
\begin{proof}

Let us first reduce to the situation in the hypothesis of Theorem \ref{thm:approx} (this is just a technical step). Let $\mathcal{P}\subset \mathbb{R}[X_0, \ldots, X_n]$ be the finite set of polynomials obtained by homogenizing the polynomials from $\mathcal{F}$; notice that the degrees of the elements from $\mathcal{P}$ are still bounded by $\delta$. Denote by $y:=[1, 0, \ldots, 0]\in \mathbb{R}\mathrm{P}^n$, $H:=\{X_0\neq 0\}=y^\perp$ and consider the affine chart 
\be \varphi:\mathbb{R}\mathrm{P}^n\setminus H\to \mathbb{R}^n,\ee
given by $\varphi([x_0, \ldots, x_n])=(\frac{x_1}{x_0}, \ldots, \frac{x_n}{x_0}).$ 

Since $V=Z(\mathcal{P}, \mathbb{R}\mathrm{P}^n\setminus H)$ has dimension at most $m$, we are in the position of applying part (A) of Theorem \ref{thm:approx}: we get a one parameter family of real algebraic sets $\{V_t\}_{t>0}$ and $t_0>0$  such that for every $0<t<t_0$ the set $V_t$ is a nonsingular complete intersection in  $\mathbb{R}^{n}$ defined by $n-m$ polynomials $P_1, \ldots, P_{n-m}\in \mathbb{R}[X_1, \ldots, X_n]$ of degree bounded by $2\delta.$

Fix $\varepsilon, \sigma>0$ and pick $R>0$ such that $B(p, \varepsilon+\sigma)\subseteq B(0, R).$ The family $\{V_t\}_{t>0}$ that we obtained applying Theorem \ref{thm:approx} satisfies:
\be \lim_{t\to 0}\left(V_t\cap B(0, R)\right)=V\cap B(0,R).\ee
We are now in the position of using Theorem \ref{thm:metric} with the choices $X=\mathbb{R}^n$, $C=V$, $C_t=V_t$ and $B=B(0, R).$
For every $\tau>0$ there exists $t_\tau>0$ such that for all $0<t<t_\tau$:
\be\label{eq:inc11} \mathcal{U}(V, \varepsilon)\cap B(p, \sigma)\subseteq \mathcal{U}(V_t, \varepsilon+\tau)\cap B(p, \sigma).\ee
Therefore, using \eqref{eq:inc11}, for every $\tau>0$ and for $0<t<\min\{t_\tau, t_0\}$ we can bound the probability in the statement by:
\begin{align} \mathbb{P}\left(\mathrm{dist}(x, V)\leq \varepsilon\right)&=\frac{\mathrm{vol}(\mathcal{U}(V, \varepsilon)\cap B(p, \sigma))}{\mathrm{vol}(B(p, \sigma))}\\\label{eq:inter}
&\leq \frac{\mathrm{vol}(\mathcal{U}(V_t, \varepsilon+\tau)\cap B(p, \sigma))}{\mathrm{vol}(B(p, \sigma))}\\
&=\mathbb{P}\left(\mathrm{dist}(x, V_t)\leq \varepsilon+\tau\right).\end{align}

Since for $t<\min\{t_\tau, t_0\}$ the set $V_t$ is a nonsingular complete intersection defined by $n-m$ polynomials of degree bounded by $2\delta,$ we can use Theorem \ref{thm:Lotz} (with the choices $d=2\delta$, $r=\varepsilon+\tau$ and $s=\sigma$) and get:
\be\mathbb{P}\left(\mathrm{dist}(x, V_t)\leq \varepsilon+\tau\right)\leq  4\sum_{i=0}^m{{n}\choose{n-m+i}}\left(\frac{4\delta(\varepsilon+\tau)}{\sigma}\right)^{n-m+i}\left(1+\frac{\varepsilon+\tau}{\sigma}\right)^{m-i}.\ee

Together with \eqref{eq:inter}, this proves that for every $\varepsilon,\sigma, \tau>0$
\be \mathbb{P}\left(\mathrm{dist}(x, V)\leq \varepsilon\right)\leq 4\sum_{i=0}^m{{n}\choose{n-m+i}}\left(\frac{4\delta(\varepsilon+\tau)}{\sigma}\right)^{n-m+i}\left(1+\frac{\varepsilon+\tau}{\sigma}\right)^{m-i}.\ee
Letting $\tau\to 0$ in the right hand side of the previous equation gives \eqref{eq:affine}.

In order to get \eqref{eq:affine2} we first estimate:
\be {n\choose n-m+i}=\underbrace{\frac{n!}{m!}}_{\leq n^{n-m}}\underbrace{\frac{i!}{(n-m+i)!}}_{\leq 1}{m\choose i}\leq n^{n-m}{m\choose i}.\ee
Using this estimate we obtain
\begin{align} \mathbb{P}\left(\mathrm{dist}(x, V)\leq \varepsilon\right)&\leq 4n^{n-m}\sum_{i=0}^m{m \choose i}\left(\frac{4\delta\varepsilon}{\sigma}\right)^{n-m+i}\left(1+\frac{\varepsilon}{\sigma}\right)^{m-i}\\
&=4 \left(\frac{4n\delta\varepsilon}{\sigma}\right)^{n-m}\left(1+\frac{(4\delta+1)\varepsilon}{\sigma}\right)^m.
\end{align}
If moreover $\varepsilon\leq \frac{\sigma}{(4\delta+1)m}$, then $\left(1+\frac{(4\delta+1)\varepsilon}{\sigma}\right)^m\leq\left(1+\frac{1}{m}\right)^m\leq e$ and \eqref{eq:affine3} follows.
\end{proof}
\begin{remark}\label{rem:comteyomdin}Let us compare the bounds from the previous theorem with the bounds that one can get using the work of Comte and Yomdin \cite{ComteYomdin}. More precisely, for a bounded definable set $V\subset \mathbb{R}^n$ one defines $M(V, \varepsilon)$ as the minimal number of $\varepsilon$--balls needed to cover $V$, so that if $V\subseteq \bigcup_{i=1}^\nu B(x_i, \varepsilon)$ then $\mathcal{U}(V, \varepsilon)\subseteq \bigcup_{i=1}^\nu B(x_i, 2\varepsilon)$ and
\be \mathrm{vol}(\mathcal{U}(V, \varepsilon))\leq (2\varepsilon)^n\mathrm{vol}(B(0,1))M(V, \varepsilon).\ee
If $V$ is definable and of dimension $m$, then \cite[Corollary 5.7]{ComteYomdin} proves that:
\be M(V\cap B(0, R), \varepsilon)\leq a(n)\sum_{i=0}^mB_{0, n-i}(V)\mathrm{vol}(B_{\mathbb{R}^i}(0, 1))\left(\frac{R}{\varepsilon}\right)^i,\ee
where $B_{0, n-i}(V)=\sup_{L}b_0(V\cap L)$ and the supremum is over all the affine planes $L\subset \mathbb{R}^n$ of dimension $n-i.$ The constant $a(n)$ can be estimated, using \cite[Theorem 3.5]{ComteYomdin}, by
\be a(n)\leq n\pi^{\frac{n-1}{2}}2^{n+\frac{n}{2}}n!(n+1)!^{\frac{1}{2}}.\ee
(This comes after some long computations that we do not reproduce here.)
When $V$ is of dimension $m$, defined by polynomials of degree bounded by $\delta$, we have $B_{0, n-i}(V)\leq (2\delta)^{n-i}$. Using these bounds, and setting $c=n-m$, one can show that the probability in \eqref{eq:affine2} can be estimated by:
\be \mathbb{P}\left(\mathrm{dist}(x, V)\leq \varepsilon\right)\leq n\pi^{\frac{n-1}{2}}2^{n+\frac{n}{2}}n!(n+1)!^{\frac{1}{2}}\Gamma\left(\frac{c}{2}\right)\left(\frac{2\delta\varepsilon}{\sigma}\right)^{c}\left(1+\frac{(4\delta+1)\varepsilon}{\sigma}\right)^m. \ee
\end{remark}

\section{The spherical case}
\label{sec:spherical}
\subsection{Preliminaries}\label{sec:preliminaries} Given a smooth submanifold $M\hookrightarrow S^n$ of dimension $m$, we denote by $TM$ its  tangent bundle and by $N^1M$ the unit normal bundle of $M$ in $S^n$, i.e.
\be N^{1}M:=\{(x,\nu)\in M\times \mathbb{R}^{n+1}\,|\, \nu\perp T_xM, \, \|\nu\|=1\}\stackrel{p_1}{\longrightarrow} M,\ee
where $p_1(x,\nu)=x.$ The metric on $N^1M$ comes by restricting the fiberwise standard metric of $M\times \mathbb{R}^n$, and similarly for $TM$. In this way for every $x\in M$ the fiber $N_x^1M$ is isometric to the standard unit sphere $S^{n-m-1}.$ 

Letting $E$ be the orientation line bundle\footnote{The reader unfamiliar with the notion of density can assume that $M$ is orientable and read this paragraph simply substituting the word ``density'' with the word ``form''. We refer to \cite[Chapter 1, \S 7 ]{BottTu} for more details.} on $M$, $\omega_M\in \Omega^{m}(M, E)$ the volume density of $M$ and $x\in M$ the variable in $M$, we denote the integration with respect to $\omega_M$ by ``$\omega_M(\mathrm{d}x)$''. We denote by $\Theta\in \Omega^{n-m-1}(N^1M, E)$ the global angular density, i.e. the density that restricts to the volume density of $N_x^1M\simeq S^{n-m-1}$ for every $x\in M$. Letting $\theta\in S^{n-m-1}$ be the variable on the sphere, we denote the integration with respect to the volume density $\Theta|_{N_xM}$ by ``$\Theta_x(\mathrm{d}\theta)$''.  A volume density on $N^1M$ is defined by 
\be\label{eq:volN} \omega_{N^1M}:=p_1^*(\omega_M)\wedge \Theta.\ee 
We denote the integration with respect to this density by ``$\omega_{N^1M}(\mathrm{d}\nu)$''. We remark that, when dealing with a density $\omega_M\in \Omega^{m}(M, E)$, given vectors $v_1, \ldots, v_m\in T_xM$, the condition $|\omega(v_1, \ldots, v_m)|=1$ is well defined (however the ``sign'' of the density is not, unless the orientation bundle $E$ is trivial and a trivialization has been chosen).

For every $x\in M$ and $\nu\in N_x^1M$ we denote by $L_x(\nu):T_xM\to T_xM$ the Weingarten map of $M$ in $S^n$ in the direction of $\nu$. For $i=0, \ldots, m$ we define the functions $\psi_i:N^1M\to \mathbb{R}$ by:
\be \det(\mathbf{1}-tL_x(\nu))=\sum_{i=0}^mt^i\psi_i(\nu).\ee

\begin{remark}\label{rem:integral}In the sequel we will use the following  fact from differential geometry. Let $L$ be a smooth manifold of dimension $\ell$ and $\omega_L$ be a density on it. Given an embedding $\gamma:L\to \mathbb{R}^{n+1}$ the $\ell$--dimensional volume of the image $\gamma(L)$ can be written as 
\be\label{eq:coarea} \mathrm{vol}(\gamma(L))=\int_{L}\rho(y)\omega_L(\mathrm{d}y),\ee where the function $\rho:L\to \mathbb{R}$ is computed as follows. For every point $y\in L$ we pick a basis $\{v_1, \ldots, v_\ell\}$ of $T_yL$ such that $|\omega_L(v_1, \ldots, v_\ell)|=1$ and we consider the matrix $J_y\gamma:=[d_y\gamma v_1, \ldots, d_y\gamma v_\ell]$. Then:
\be \rho(y)=\sqrt{\det\left(J_y\gamma^TJ_y\gamma\right)}.\ee
\end{remark}

\begin{definition}Let $M\hookrightarrow S^n$ be a smooth manifold of dimension $m$, possibly with boundary, and $A\subset M\setminus \partial M$ be an open set. For every integer $0\leq i\leq m$ we define $I_i:A\to \mathbb{R}$ by
\be I_i(x):=\int_{N^1_xM}|\psi_i(\theta)|\,\Theta_{x}(\mathrm{d}\theta).\ee
The $i$--th total absolute curvature of $A\subseteq M\setminus \partial M$ is defined by:
\be |K_i|(A):=\int_{A}I_i(x)\, \omega_{M}(\mathrm{d}x)=\int_{N^1M}|\psi_i(\nu)|\omega_{N^1M}(\mathrm{d}\nu).\ee
(The right hand side equality follows from \eqref{eq:volN}.)
\end{definition}
\begin{definition}\label{defi:J}For every pair $(k,n)$  of natural numbers with $k\leq n$ we define the function $J_{n,k}:[0, \infty)\to \mathbb{R}$ by
\be \label{eq:J} J_{n, 0}\equiv 1\quad \textrm{and}\quad J_{n,k}(\varepsilon):=\int_{0}^{\min\{\varepsilon, \frac{\pi}{2}\}} (\sin \theta)^{k-1}(\cos \theta)^{n-k}\mathrm{d}\theta.\ee
(Note the extremum of integration in the definition of the function $J_{n,k}$ for $k>0$.)
\end{definition}
Let us recall some useful properties of the functions $J_{n,k}$. 
\begin{lemma}\label{lemma:pJ}For $0\leq\sigma, \varepsilon\leq \frac{\pi}{2}$ and $p\in S^n$, we have the following properties:
\begin{enumerate}
\item $J_{n,k}(\varepsilon)\leq \frac{(\sin \varepsilon)^k}{k}$ for $k\neq n$.
\item $J_{n,n}(\varepsilon)\leq \frac{1}{2}\mathrm{vol}(S^n)(\sin \varepsilon)^n.$
\item $\mathrm{vol}(B(p, \sigma))=\mathrm{vol}(S^{n-1})J_{n,n}(\sigma)\geq \mathrm{vol}(S^{n-1})\frac{(\sin \sigma)^n}{n}$.
\end{enumerate}
\end{lemma}
\begin{proof}These are reformulations in our notation of \cite[Lemma 2.31]{BuCu}, \cite[Lemma 2.34]{BuCu} and \cite[Lemma 20.6]{BuCu} respectively.
\end{proof}
\begin{definition}Let $M\hookrightarrow S^n$ be a smooth manifold, possibly with boundary, and $A\subseteq M\setminus \partial M$ be an open set. For $\varepsilon>0$, we denote by $\mathcal{T}(A, \varepsilon)$ the $\varepsilon$--tube of $A$ in $S^n$, i.e. the set of points $p\in S^n$ such that there exists a geodesic of length at most  $\varepsilon$ on $S^n$, joining $p$ with $A$ and meeting $A$ orthogonally. If we we denote by $\mathrm{exp}:TS^n\to S^n$ the riemannian exponential map, and by $N^\varepsilon A\subset NA\subset TS^n$ the set of vectors in $NA$ of norm at most $\varepsilon$, for small enough $\varepsilon>0$ we have:
\be \mathcal{T}(A, \varepsilon)=\mathrm{exp}\left(N^{\varepsilon}A\right).\ee
\end{definition}
\begin{remark}The function $J_{n,k}$ play the analogues of polynomials for the spherical version of Weyl's Tube Formula. More precisely, given a smooth compact manifold $M\hookrightarrow S^n$, of dimension $m$ and codimension $c=n-m$, possibly with boundary, and an open set $A\subseteq M\setminus \partial M$, there exists $\varepsilon_0>0$ such that for all $0<\varepsilon<\varepsilon_0$ we have \cite{Weyl}:
\be\label{eq:weyl}\mathrm{vol}(\mathcal{T}(A, \varepsilon))=\sum_{{i=0}}^mJ_{n, c+i}(\varepsilon)\int_{N^1M}\psi_i(\nu)\omega_{N^1M}(\mathrm{d}\nu).\ee
Since $\psi_i(\nu)$ is a homogeneous polynomial of degree $i$, its integral on $N_{x}^1M$ vanishes for $i$ odd. The integral $\int_{N^1M}\psi_i(\nu)\omega_{N^1M}(\mathrm{d}\nu)$ is usually denoted by $K_i(M)$ and called $i$--th curvature integral.
\end{remark}
\subsubsection{Volume of tubes and curvature integrals}
We prove now a sequence of useful lemmas, relating the volume of tubes and the total absolute curvatures.
\begin{lemma}\label{lemma1}Let $M\hookrightarrow S^n$ be a smooth manifold, possibly with boundary, of dimension $m$ and codimension $c=n-m$. Let $A\subseteq M\setminus \partial M$ be an open set. For every $0\leq\varepsilon\leq \frac{\pi}{2}$:
\be \mathrm{vol}(\mathcal{T}(A, \varepsilon)) \leq \sum_{{i=0}}^mJ_{n, c+i}(\varepsilon) |K_i|(A).\ee
\end{lemma}
\begin{proof}This is simply an adaptation of \cite[Theorem 3.1]{Lotz} to the spherical case. We prove the case $\varepsilon<\frac{\pi}{2}$; the case $\varepsilon=\frac{\pi}{2}$ follows easily by a limit argument. 

Consider the density $r^{n-m-1} \mathrm{d}r\wedge \Theta$ on $NM$ and the map $\varphi:NM\simeq N^1M\times (0, \infty)\to S^n$ defined by
\be \varphi(\nu, r):=\mathrm{exp}_{p(\nu)}(rv)=\frac{p(\nu)+r \nu}{(1+r^2)^{\frac{1}{2}}}.\ee
The Jacobian of the map $\varphi$ is computed in \cite{Weyl} and can be used to compute the pull-back of the volume density $\omega_{S^n}$ under $\varphi$:
\be\varphi^*(\omega_{S^n}) =\frac{|\det(\mathbf{1}-rL_x(\nu))|}{(1+r^2)^{\frac{n+1}{2}}}r^{n-m-1}dr\wedge \Theta.
\ee 
The image of $\varphi|_{N^1A\times (0, \tan\varepsilon)}$ contains $\mathcal{T}(A, \varepsilon)$ and by the change of variables formula, since the set of critical values of $\varphi$ has measure zero:
\begin{align} \mathrm{vol}(\mathcal{T}(A, \varepsilon))&\leq \int_{N^1A\times (0, \tan\varepsilon)}\frac{|\det(\mathbf{1}-rL_x(\nu))|}{(1+r^2)^{\frac{n+1}{2}}}r^{n-m-1}dr\wedge \Theta(\mathrm{d}\nu)\\
&=\int_{N^1A\times (0, \tan\varepsilon)}\frac{\left|\sum_{i=0}^mr^i\psi_i(\nu)\right|}{(1+r^2)^{\frac{n+1}{2}}}r^{n-m-1}dr\wedge \Theta(\mathrm{d}\nu)\\
&\leq\int_{N^1A\times (0, \tan\varepsilon)}\sum_{i=0}^m|\psi_i(\nu)|\wedge \Theta(\mathrm{d}\nu) \\
&=\sum_{i=0}^m\left( \int_{0}^{\tan\varepsilon}\frac{r^{n-m-1+i}}{(1+r^2)^{\frac{n+1}{2}}}dr\right)\left(\int_{N^1A}|\psi_i(\nu)|\omega_{N^1M}(\mathrm{d}\nu)\right)\\
&=\sum_{i=0}^m J_{n, c+i}(\varepsilon)|K_i|(A).\end{align}
\end{proof}
\begin{definition}For $n\in \mathbb{N}$ we denote by $O(n)\subset \mathbb{R}^{n\times n}$ the orthogonal group, with the induced riemannian structure, and by ``$\mathrm{d}g$'' the integration with respect to the corresponding volume density. If $f:O(n)\to \mathbb{R}$ is a measurable function we define
\be \underset{g\in O(n)}{\mathbb{E}} f(g):=\frac{1}{\mathrm{vol}(O(n))}\int_{O(n)}f(g)\,\mathrm{d} g.\ee
\end{definition}
\begin{lemma}\label{lemma2}Let $M\hookrightarrow S^n$ be a smooth manifold, possibly with boundary, of dimension $m$ and codimension $c=n-m$. Let $A\subseteq M\setminus \partial M$ be an open set. Then for every $0\leq i\leq m$ we have:
\be K_i(M)=\frac{\pi^{\frac{1}{2}}\Gamma\left(\frac{n}{2}\right)}{\Gamma\left(\frac{m-i+1}{2}\right)\Gamma\left(\frac{n-m+i}{2}\right)}\underset{g\in O(n+1)}{\mathbb{E}}K_i(M\cap g\cdot S^{n-m+i}).\ee
\end{lemma}
\begin{proof}This is simply a restatement of \cite[Theorem A.59]{BuCu} in our notation.
\end{proof}

\begin{lemma}\label{lemma3}Let $M\hookrightarrow S^n$ be a smooth manifold, possibly with boundary, of dimension $m$ and codimension $c=n-m$. Let $A\subseteq M\setminus \partial M$ be an open set. Then for every $0\leq i\leq m$ we have:
\be\label{eq:averageabs} |K_i|(M)\leq 2\frac{\pi^{\frac{1}{2}}\Gamma\left(\frac{n}{2}\right)}{\Gamma\left(\frac{m-i+1}{2}\right)\Gamma\left(\frac{n-m+i}{2}\right)}\underset{g\in O(n+1)}{\mathbb{E}}|K_i|(M\cap g\cdot S^{n-m+i}).\ee
\end{lemma}
\begin{proof}This is an adaptation of \cite[Theorem 3.3]{Lotz} to the spherical case. One denotes by $A_+$ and $A_-$ the set of points in $A$ where $I_i$ is positive and negative, respectively. Then  $|K_i|(A)=|K_i(A_+)|+|K_i(A_-)|\leq 2|K_i(A)|.$ The inequality \eqref{eq:averageabs} follows now from Lemma \ref{lemma2}.
\end{proof}
\subsubsection{The gauss map}
\begin{definition}\label{defi:gauss}Let $Y\hookrightarrow S^{n}$ be a smooth manifold and $N^1Y$ be its unit normal bundle. Observe that $N^1Y=\{(y, \nu)\in S^n\times S^n\,|\, y\in Y, \nu\in (T_yY)^\perp\}\subset S^n\times S^n$. We denote by 
\be\gamma_Y:=p_2|_{N^1Y}:N^1Y\to S^n\ee the restriction of the projection on the second factor and call it the \emph{Gauss map} of $Y$.
\end{definition}

\begin{proposition}\label{esti}Let $Y\hookrightarrow S^n$ be a smooth semialgebraic manifold of dimension $i$. Let $\gamma=\gamma_Y:N^1Y\to S^n$ be the corresponding Gauss map. Consider the set $O_{\gamma, \pitchfork}:=\{h\in O(n+1)\,|\, \gamma \pitchfork h\cdot S^{1}\}.$\footnote{Given a map $\gamma:A\to B $ between smooth manifolds and a submanifold $S\hookrightarrow B$, the symbol ``$\gamma \pitchfork S$'' stands for ``$\gamma$ is transversal to $S$'', i.e. $\mathrm{im}(d_x\gamma)+T_{\gamma(x)}S=T_{\gamma(x)}B$ for every $x\in A$ such that $\gamma(x)\in S$.}
Then $O_{\gamma, \pitchfork}\subseteq O(n+1)$ has full measure and:
\be |K_i|(Y)\leq \frac{\mathrm{vol}(S^{n-1})}{2}\underset{h\in O_{\gamma, \pitchfork}}{\sup}\#\left(\gamma^{-1}(h\cdot S^{1})\right).\ee
\end{proposition}
\begin{proof}

Let $p_1:N^1Y\to Y$ and $p_2:N^1Y\to S^n$ be the projections onto the two factors (recall that $\gamma_Y=p_2$). Pick a point $(y, \nu)\in N^1Y\subset Y\times S^n$ and write
\be\label{eq:split} T_{(y, \nu)}\simeq T_yY\oplus W_y,\ee
where $W_y\subset N_yY$ is the orhtogonal complement of $\mathbb{R}\nu$ in $N_yY$ and 
with $d_{(y, \nu)}p_1$ the orthogonal projection to $T_yY$ and $d_{(y, \nu)}p_2$ the orthogonal projection to $W_y$.

Pick orthonormal bases $\{v_1, \ldots, v_i\}$ for $T_yY$ and $\{w_1, \ldots, w_{n-1-i}, \nu\}$ for $N_yY$, so that, using the identification \eqref{eq:split}, the list $\{v_1, \ldots, v_i, w_1, \ldots, w_{n-i-1}\}$ is an orthonormal basis for $T_{(y,\nu)}N^1Y$. In particular:
\be |\omega_{N^1Y}(v_1, \ldots, v_i, w_1, \ldots, w_{n-i-1})|=1,\ee
where $\omega_{N^1Y}=p_1^*\omega_Y\wedge \Theta$ is the density defined in \eqref{eq:volN}.

We show now that for every $\nu\in N^1Y$ we have:
\be |\psi_i(\nu)|=\sqrt{\det\left((J_{(y,\nu)}\gamma)^T J_{(y,\nu)}\gamma\right)},\ee
where $J_{(y,\nu)}\gamma$ is the matrix $[(d_{(y, \nu)}\gamma) v_1, \ldots, (d_{(y, \nu)})\gamma v_i, (d_{(y, \nu)})\gamma w_1, \ldots, (d_{(y, \nu)})\gamma w_{n-i-1}]$ (i.e. the columns of $J_{(y,\nu)}\gamma$ are the coordinate vectors of the images under $d_{(u, \nu)}\gamma$ of the chosen basis elements). To this end, observe first that, since $d_{(y, \nu)}p_2$ is the orthogonal projection to $W_y$,  
\be (d_{(y, \nu)}\gamma) w_j=w_j\quad \forall j=1, \ldots, n-i-1.\ee
In order to compute $ (d_{(y, \nu)}\gamma) v_j$, for $j=1, \ldots, i$, we take a curve $c:(-\epsilon,\epsilon)\to N^1Y$ such that $c(0)=(y, \nu)$ and $\dot{c}(0)=v_j$. Notice that $\nu(t):=p_2(c(t))$ defines a normal field on $Y$ along $c$ and with $\nu(0)=\nu$. In particular, denoting by $L_y(\nu):T_yY\to T_yY$ the Weingarten map in the direction of $\nu$, we have :
\be (d_{(y, \nu)}\gamma) v_j=\frac{d}{dt}\left(p_2(c(t))\right)\big|_{t=0}=\frac{d}{dt}\left(\nu(t))\right)\big|_{t=0}=\nabla_{\frac{d}{dt}}\nu(t)\big|_{t=0}=L_y(\nu)v_j.\ee
This shows that the matrix $(J_{(y,\nu)}\gamma)^T J_{(y,\nu)}\gamma$ has the following shape:
\be (J_{(y,\nu)}\gamma)^T( J_{(y,\nu)}\gamma)=\left(\begin{array}{c|c}Q & 0 \\\hline 0 & \mathbf{1}\end{array}\right), \ee
where $Q_{ij}=v_i^TL_y(\nu)^TL_y(\nu)v_j$. In particular, as claimed:
\be\label{eq:det} \sqrt{\det\left((J_{(y,\nu)}\gamma)^T J_{(y,\nu)}\gamma\right)}=|\det(L_y(\nu))|=|\psi_i(\nu)|.\ee

Observe that if $E\subseteq N^1Y$ is such that $\gamma|_E$ is an embedding, using Remark \ref{rem:integral}, we get:
\be \int_E|\psi_i(\nu)|\omega_{N^1Y}(\mathrm{d}\nu)=\mathrm{vol}_{n-1}(\gamma(E)).\ee

Since $Y$ is semialgebraic, so are $N^1Y$ and $\gamma$ and we can partition $N^1Y=E_1\sqcup E_2$ into semialgebraic pieces such that $\mathrm{rk}(d_\nu\gamma)=n-1$ for $\nu\in E_1$ and $\mathrm{rk}(d_\nu\gamma)\leq n-2$ for $\nu\in E_2.$ By \eqref{eq:det} we have $|\psi_i|\equiv 0$ on $E_2$ and therefore:
\be |K_i|(Y)=\int_{N^1Y}|\psi_i(\nu)|\omega_{N^1Y}(\mathrm{d}\nu)=\int_{E_1}|\psi_i(\nu)|\omega_{N^1Y}(\mathrm{d}\nu).\ee

Using again the semialgebraic assumptions, we can partition $E_1=\left(\bigsqcup_{j=1}^a E_{1,j}\right)\sqcup E_{1,0}$ into semialgebraic pieces such that for every $j=1, \ldots, a$ the set $E_{1,j}\hookrightarrow N^1Y$ is a smooth submanifold of dimension $n-1$, $\gamma|_{E_{1,j}}$ is an embedding, and $E_{1,0}$ is of dimension smaller then or equal to $n-2$ (and in particular it has measure zero). Then
\begin{align}\int_{N^1Y}|\psi_i(\nu)|\omega_{N^1Y}(\mathrm{d}\nu)&=\sum_{j=1}^a\int_{E_{1,j}}|\psi_i(\nu)|\omega_{N^1Y}(\mathrm{d}\nu)\\
&=\sum_{j=1}^a\mathrm{vol}_{n-1}(\gamma(E_{1, j}))=(*).
\end{align}
Since $\gamma(E_{i,j})$ is a submanifold of $S^{n}$ of dimension $n-1$, we have the kinematic identity \cite{Howard}:
\be\label{eq:kine} \underset{h\in O(n+1)}{\mathbb{E}}\#\left(\gamma(E_{1,j})\cap h\cdot S^{1}\right)=2\frac{\mathrm{vol}(\gamma(E_{1,j}))}{\mathrm{vol}(S^{n-1})}.\ee
Using \eqref{eq:kine} we can continue with
\begin{align}(*)&=\sum_{j=1}^a\frac{\mathrm{vol}(S^{n-1})}{2}\underset{h\in O(n+1)}{\mathbb{E}}\#\left(\gamma(E_{1,j})\cap h\cdot S^{1}\right)\\
&=\sum_{j=1}^a\frac{\mathrm{vol}(S^{n-1})}{2}\underset{h\in O(\ell)}{\mathbb{E}}\#\left(\gamma|_{E_{1,j}}^{-1}( h\cdot S^{1})\right)\\
&=\frac{\mathrm{vol}(S^{n-1})}{2}\underset{h\in O(n+1)}{\mathbb{E}}\#\left(\gamma|_{E_1\setminus E_{0,1}}^{-1}( h\cdot S^{1})\right)\\
&= \frac{\mathrm{vol}(S^{n-1})}{2}\underset{h\in O_{\gamma, \pitchfork}}{\mathbb{E}}\#\left(\gamma|_{E_1\setminus E_{0,1}}^{-1}( h\cdot S^{1})\right)\\
&\leq \frac{\mathrm{vol}(S^{n-1})}{2}\underset{h\in O_{\gamma, \pitchfork}}{\mathbb{E}}\#\left(\gamma^{-1}( h\cdot S^{1})\right)\\
&\leq \frac{\mathrm{vol}(S^{n-1})}{2}\underset{h\in O_{\gamma, \pitchfork}}{\sup}\#\left(\gamma^{-1}(h\cdot S^{1})\right).
\end{align}
\end{proof}
Motivated by the previous result we introduce the following definition.
\begin{definition}\label{def:bi}Let $M\hookrightarrow S^n$ be a smooth manifold of dimension $m$. For every $0\leq i\leq m$ let $O_{M, \pitchfork}:=\{g\in O(n+1)\,|\, M\pitchfork g\cdot S^{n-m+i}\}$ and for every $g\in O_{M, \pitchfork}$ consider the set $O_{g, \pitchfork}:= \{h\in O(n+1)\,|\, \gamma_{ M\cap g\cdot S^{n-m+i}}\pitchfork h\cdot S^{1}\}$. We define:
\be\beta_i(M):=\underset{g\in O_{M, \pitchfork}}{\sup}\,\underset{h\in O_{g, \pitchfork}}{\sup} \#\left(\gamma_{M\cap g\cdot S^{n-m+i}}^{-1}(h\cdot S^{1})\right).\ee
\end{definition}
Before proving next result, we will need the following technical lemma.
\begin{lemma}\label{lemma:angle}Let $B(p, \sigma)\subset S^{n}$ be a ball with $\sigma>0$. Then
\be \mathbb{P}\left(g\cdot S^{n-m+i}\cap B(p, \sigma)\neq \emptyset\right)=\frac{2\Gamma\left(\frac{n+1}{2}\right)}{\Gamma\left(\frac{n-m+i+1}{2}\right)\Gamma\left(\frac{m-i}{2}\right)}J_{n, m-i}(\sigma).\ee
\end{lemma}
\begin{proof}Let $c=n-m$. Observe first that
\be \label{eq:P1}\mathbb{P}\left(g\cdot S^{c+i}\cap B(p, \sigma)\neq \emptyset\right)=\mathbb{P}\left( S^{c+i}\cap B(g^{-1}\cdot p, \sigma)\neq \emptyset\right)=\mathbb{P}\left( S^{c+i}\cap B(g\cdot p, \sigma)\neq \emptyset\right),\ee
since $g^{-1}\in O(n+1)$ is still uniformly distributed. Denoting by $\ell=\mathrm{span}\{p\}\subset \mathbb{R}^{n+1}$, we see that $g\cdot \ell$ is a uniformly distributed one dimensional linear space and, denoting by $0\leq\theta_1(\ell, W)\leq\frac{\pi}{2} $ the first principal angle between $\ell$ and a $(c+i+1)$--dimensional linear space $W$, for $\sigma\leq \frac{\pi}{2}$ we get:
\be \mathbb{P}\left( S^{c+i}\cap B(gp, \sigma)\neq \emptyset\right)=\mathbb{P}\left(\theta_1(g\cdot \ell, \mathbb{R}^{c+i+1})\leq \sigma\right).\ee
The density $p(\theta)$ of $\theta_1$ is computed in \cite[Theorem 3.2]{PSC} and it is given by:
\be p(\theta)=\frac{2\Gamma\left(\frac{n+1}{2}\right)}{\Gamma\left(\frac{n-m+i+1}{2}\right)\Gamma\left(\frac{m-i}{2}\right)}(\cos \theta)^{c+i}(\sin\theta)^{n-c-i-1}.\ee
Integrating this function between $0$ and $\sigma$ gives the desired probability in the case $\sigma\leq \frac{\pi}{2}$:
\be \label{eq:ff}\mathbb{P}\left(g\cdot S^{n-m+i}\cap B(p, \sigma)\neq \emptyset\right)=\int_0^{\sigma}p(\theta)\mathrm{d}\theta=\frac{2\Gamma\left(\frac{n+1}{2}\right)}{\Gamma\left(\frac{n-m+i+1}{2}\right)\Gamma\left(\frac{m-i}{2}\right)}J_{n, m-i}(\sigma).\ee

If $\sigma>\frac{\pi}{2}$, then $\mathbb{P}\left(g\cdot S^{c+i}\cap B(p, \sigma)\neq \emptyset\right)=1$. Recall now that we have defined $J_{n, m-i}(\sigma)=\int_{0}^{\min\{\sigma, \frac{\pi}{2}\}}(\cos \theta)^{c+i}(\sin\theta)^{n-c-i-1}$. In particular, since
\be \frac{2\Gamma\left(\frac{n+1}{2}\right)}{\Gamma\left(\frac{n-m+i+1}{2}\right)\Gamma\left(\frac{m-i}{2}\right)}J_{n, m-i}\left(\frac{\pi}{2}\right)=1,\ee
then \eqref{eq:ff} is still valid for $\sigma>\frac{\pi}{2}.$ 
\end{proof}
\begin{proposition}\label{propo:intermediate}Let $M\hookrightarrow S^n$ be a smooth manifold of dimension $m$, $p\in S^n$, $\sigma>0$ and $A\subseteq M\setminus \partial M$ be an open set contained in $B(p, \sigma).$ Then for every $0\leq\varepsilon\leq \frac{\pi}{2}:$
\begin{align}\mathrm{vol}(\mathcal{T}(A, \varepsilon))&\leq \mathrm{vol}(S^{n-1})\sum_{i=0}^m{{n-1}\choose{m-i}}(m-i)J_{n, c+i}(\varepsilon)J_{n, m-i}(\sigma)\beta_i(M).
\end{align}
\end{proposition}
\begin{proof}Let us set $\alpha_i(n,m):=\frac{\pi^{\frac{1}{2}}\Gamma\left(\frac{n}{2}\right)}{\Gamma\left(\frac{m-i+1}{2}\right)\Gamma\left(\frac{n-m+i}{2}\right)}.$ Using the above results we have:
\begin{align}\mathrm{vol}(\mathcal{T}(A, \varepsilon)) &\leq \sum_{{i=0}}^mJ_{n, c+i}(\varepsilon) |K_i|(A)&\textrm{(Lemma \ref{lemma1})}\\
&\label{eq:pro2}\leq 2\sum_{{i=0}}^mJ_{n, c+i}(\varepsilon) \alpha_i(n,m)\underset{g\in O(n+1)}{\mathbb{E}}|K_i|(A\cap g\cdot S^{n-m+i})&\textrm{(Lemma \ref{lemma3})}\end{align}
Observe now that, by Proposition \ref{esti},
\begin{align} \underset{g\in O(n+1)}{\mathbb{E}}|K_i|(A\cap g\cdot S^{n-m+i})&\leq \underset{g\in O(n+1)}{\mathbb{E}}\frac{\mathrm{vol}(S^{n-1})}{2}\sup_{h\in O_{g, \gamma}}\#\left((\gamma_{M\cap g\cdot S^{n-m+i}})^{-1}(h\cdot S^{1})\right)\\
&\label{eq:pro1}\leq \frac{\mathrm{vol}(S^{n-1})}{2}\beta_i(M)\mathbb{P}\left(g\cdot S^{n-m+i}\cap B(p, \sigma)\neq \emptyset\right). \end{align}
The probability on the right hand side of \eqref{eq:pro1} is computed in Lemma \ref{lemma:angle} (which has no restriction on $\sigma>0$, using the convention \eqref{eq:J}). Substituting \eqref{eq:pro1} into \eqref{eq:pro2} we get:
\be \mathrm{vol}(\mathcal{T}(A, \varepsilon))\leq \mathrm{vol}(S^{n-1})\sum_{i=0}^m {\alpha_{i}(n,m)}{\frac{2\Gamma\left(\frac{n+1}{2}\right)}{\Gamma\left(\frac{n-m+i+1}{2}\right)\Gamma\left(\frac{m-i}{2}\right)}}J_{n, c_i}(\varepsilon)J_{n, m-i}(\sigma)\beta_i(M).\ee
The inequality in the statement  follows from the identity:
\be {\alpha_{i}(n,m)}{\frac{2\Gamma\left(\frac{n+1}{2}\right)}{\Gamma\left(\frac{n-m+i+1}{2}\right)\Gamma\left(\frac{m-i}{2}\right)}}=\frac{\Gamma(n)}{\Gamma(m-i)\Gamma(n-m+i)}={{n-1}\choose{m-i}}(m-i).\ee
\end{proof}
\subsubsection{The spherical algebraic case: degree estimates}Next lemma estimates the quantities $\beta_i(Z)$ defined in Definition \ref{def:bi} in the case $Z\subset S^n$ is a smooth complete intersection.

\begin{lemma}\label{lemma:degree}Let $Z$ be the zero set of homogeneous polynomials $P_1, \ldots, P_c$ of degree at most $d$ in $S^n$. Assume that $Z$ is a non-singular complete intersection of dimension $m=n-c$. Then for every $i=0, \ldots, m$ we have:
\be \beta_i(Z)\leq 
 2(4 d) ^{n-m+i}.
\ee
\end{lemma}

\begin{proof}
Let $X_0,\ldots,X_n$ be linear coordinates on $\mathbb{R}^{n+1}$ such that 
$P_1,\ldots,P_c \in \mathbb{R}[X_0,\ldots,X_n]$.
Moreover, since we are interested in an upper bound on $\beta_i(Z)$, we can assume without loss of generality
that the polynomials $P_1,\ldots,P_c$ defining $Z$ are of the same degree $d$.
The tangent bundle $TS^n$ is embedded in the tangent bundle $T\mathbb{R}^{n+1}$, which is trivial, and we denote by $Y_0, \ldots, Y_n$, the coordinate functions corresponding to the basis $(\frac{\partial}{\partial X_0}, \ldots, \frac{\partial}{\partial X_n})$ of $T_0\mathbb{R}^{n+1}\simeq \mathbb{R}^{n+1}$.

Following Definition~\ref{def:bi}, for an element $g \in O_{Z,\pitchfork}(n+1)$, the set
$Z \cap g \cdot S^{n-m+i}$ can be described by \eqref{eqn:NM1} and
\begin{align}
P_1 = \cdots = P_c =& \ 0, \\
\label{eqn:NM3}
L_0 = \cdots = L_{m-i-1} =& \ 0, \\
\end{align}
intersected with $S^n$, 
where the $L_i$'s are generic linear forms in $X_0,\ldots,X_n$.

By making a linear change in the $X_i$, we may assume that 
$L_j = X_j, 0 \leq j \leq m-i-1$.

With this choice of coordinates, 
the normal bundle 
$NZ \hookrightarrow  TS^n \hookrightarrow T\mathbb{R}^{n+1}$
can be described as projection to $T\mathbb{R}^{n+1}$,
of the solutions to the set of equations,
\begin{align}
X_{0} = \cdots = X_{m-i-1} =& \ 0,\\
\label{eqn:NM1}
P_1 = \cdots = P_c =& \ 0, \\
\nonumber
\sum_{j=1}^c \lambda_j \frac{\partial P_j}{\partial X_{m-i}} =& \ Y_{m-i}, \\
\label{eqn:NM2} \vdots& \\
\sum_{j=1}^c \lambda_j \frac{\partial P_j}{\partial X_n} =& \ Y_n,
\end{align}
intersected with $S^n \times \mathbb{R}^{c+n+1}$, 
where $\lambda_1,\ldots,\lambda_c$ are Lagrangian variables.
Notice that the above equations imply via the Euler identity that
\[
X_0 Y_0 + \cdots + X_nY_n =0,
\]
and hence $Y_0 \frac{\partial}{\partial X_0} + \cdots + Y_n \frac{\partial}{\partial X_0}$ is constrained to belong
to the tangent space of $S^n$. 

We have to bound the cardinality of $\gamma_{Z \cap g\cdot S^{n-m+i}}^{-1}(h \cdot S^{1})$ for
$h \in O_{g,\pitchfork}(n+1)$. Being zero--dimensional, we can assume (up to making another linear change of coordinates $X_{n-m},\ldots,X_n$) that
$\gamma_{Z \cap g\cdot S^{n-m+i}}^{-1}(h \cdot S^{1})$ has an empty intersection with 
the hyperplane $X_n = 0$.

Now, we can take $h\cdot S^{1}$ to be the intersection of $S^n$, with $n-1$ hyperplanes
defined by generic linear forms in $Y_0, \ldots, Y_n$, which after another linear change in
coordinates we can assume to be $Y_2,\ldots,Y_n$. 

With the above assumptions, 
using Definition \ref{defi:gauss}, the pull-back of $h \cdot S^{1}$ under 
$\gamma_{Z \cap g\cdot S^{n-m+i}}$, for $h \in O_{g,\pitchfork}(n+1)$, can be described by 
\eqref{eqn:NM1}, \eqref{eqn:NM2} and 
\be
\label{eqn:NM5}
Y_2 = \cdots = Y_n = 0.
\ee
intersected with $S^{n} \times \mathbb{R}^{c+ n+1}$ and projected to the first factor.

Homogenizing \eqref{eqn:NM2} with respect to $X_n$, we obtain the following system of
bi-homogeneous equations,
\begin{align}\label{eq:system}
P_1 = \cdots = P_c =& \ 0, \\
X_0 = \cdots = X_{m-i-1} =& \ 0, \\
\nonumber
\sum_{j=1}^c \lambda_j \frac{\partial P_j}{\partial X_{m-i}} =& \ X_n^{d-1} Y_{m-i}, \\
 \vdots& \\
\sum_{j=1}^c \lambda_j \frac{\partial P_j}{\partial X_n} =& \ X_n^{d-1} Y_n, \\
Y_2 = \cdots = Y_n =& \ 0. \\
\end{align}

These equations are homogeneous of degree  at most $d$ in $(X_0,\ldots,X_n)$, and homogeneous 
of degree at most $1$
in $(Y_0,\ldots,Y_n,\lambda_1,\ldots,\lambda_c)$. Also, using the fact
$X_0 = \cdots = X_{n-m-1} = Y_2=\cdots= Y_n = 0$, the above equations define
a zero--dimensional subvariety of 
$\mathbb{C}\mathrm{P}^{n-m+i} \times  \mathbb{C}\mathrm{P}^{c+1}$.

Using Kouchnirenko's Theorem \cite{Kouchnirenko} the number of non-degenerate roots
of this system in $\mathbb{C}\mathrm{P}^{n-m+i} \times  \mathbb{C}\mathrm{P}^{c+1}$ is bounded by
\[
(n-(m-i) + c+1)! \cdot  \mathrm{vol}_{n-(m-i) + c+1} (d \cdot \Delta_{n-(m-i)} \times \Delta_{c+1}),
\]
where for $p > 0$, 
$\Delta_p$ denotes the $p$-dimensional simplex in $\mathbb{R}^p$ (i.e. with vertices the origin and
the standard basis vectors).
Noting that
\[
\mathrm{vol}_p (\Delta_p) = \frac{1}{p!},
\]
we obtain that the number of non-degenerate roots
of this system in $\mathbb{C}\mathrm{P}^{n-m+i} \times  \mathbb{C}\mathrm{P}^{c+1}$
is bounded by 
\begin{eqnarray*}
\frac{(n-(m-i) + c+1)!}{(n-(m-i))! (c+1)!}d^{n-m+i} & = & \binom{2c +i +1}{c+1} d^{n-m+i}\\
&\leq & 2^{2c+i} d^{n-m+i} \\
&\leq& 2^{2(c+i)} d^{c+i}\\
&\leq&  (4 d) ^{c+i}.
\end{eqnarray*}
This gives a bound for the number of real projective solutions of the system of equations in \eqref{eq:system}. The lemma follows after noting that $S^n$ is a double covering of $\mathbb{R}\mathrm{P}^n$ which gives the extra factor of $2$.
\end{proof}

\subsubsection{The spherical case: smooth complete intersections}
Next theorem is the spherical analogue of Theorem \ref{thm:Lotz}.
\begin{theorem}\label{thm:completeS}Let $Z$ be the zero set of homogeneous polynomials $P_1, \ldots, P_c$ in $S^n$ of degree at most $d$. Assume that $Z$ is a smooth complete intersection of dimension $m=n-c$. Let $x$ be a uniformly distributed point in a ball $B(p, s)$ of radius $0\leq s\leq\frac{\pi}{2}$ around $p\in S^n$. Then, for every $0\leq r\leq \frac{\pi}{2}$:
\be \mathbb{P}(\mathrm{dist}(x, Z)\leq r)\leq  2 \left(1+\frac{\mathrm{vol}(S^n)}{2}\right)\left(\frac{4n d \sin r}{\sin s}\right)^{n-m}\left(1+(4nd +4d+1)\frac{\sin r}{\sin s}\right)^{m}.
\ee
In particular, if $\sin r\leq\frac{\sin s}{(4nd +4d+1)m}$, 
\be \mathbb{P}(\mathrm{dist}(x, Z)\leq r)\leq  2 e\left(1+\frac{\mathrm{vol}(S^n)}{2}\right) \left(\frac{4n d \sin r}{\sin s}\right)^{n-m}.\ee
\end{theorem}
\begin{remark}The bounds from the statements that we put in the Introduction follow from:
\be \max_{n\geq 0}\frac{\mathrm{vol}(S^n)}{2}=\frac{\mathrm{vol}(S^6)}{2}=\frac{8}{15}\pi^3\leq 20.\ee
Since $\mathrm{vol}(S^n)\to 0$ as $n\to \infty$, for large $n$ the reader might want to keep using the bound from the previous statement (similarly for Theorem \ref{thm:gS} and Theorem \ref{thm:ill}).
\end{remark}
\begin{proof}
Let $Z\hookrightarrow S^n$ be a complete intersection defined by polynomials $P_1, \ldots, P_c$ whose degrees are bounded by $d$. Let $r, s>0$ such that $Z\pitchfork B(p, s+r)$, denote by $Z'=Z\cap B(p, s+r)$ and define:
\be A_0=Z'\setminus \partial Z'\quad \textrm{and}\quad A_1=\partial Z'.\ee
With this choice we have
\be \mathcal{U}(Z, r)\cap B(p ,s)\subseteq \mathcal{T}(A_0, r)\cup \mathcal{T}(A_1, r),
\ee
and consequently
\be\label{eq:in} \mathbb{P}(\mathrm{dist}(x, Z)\leq r)=\frac{\mathrm{vol}(\mathcal{U}(Z, r)\cap B(p, s))}{\mathrm{vol}(B(p, s))}\leq\frac{\mathrm{vol}(\mathcal{T}(A_0, r))}{\mathrm{vol}(B(p, s))}+\frac{\mathrm{vol}(\mathcal{T}(A_1, r))}{\mathrm{vol}(B(p, s))}.\ee
We apply now Proposition \ref{propo:intermediate} (which has no constraints on $\sigma>0$) for estimating both summands of \eqref{eq:in}, with the choice $\sigma=s+r$, $\varepsilon=r$ and $A=A_0, A_1.$
For $j=0, 1$, using the properties from Lemma \ref{lemma:pJ} and setting $(r+s)':=\min\{r+s, \frac{\pi}{2}\}$, we get:
\begin{align}\frac{\mathrm{vol}(\mathcal{T}(A_j, r))}{\mathrm{vol}(B(p, s))}\leq&\frac{1}{J_{n,n}(s)}\sum_{i=0}^{m-1-j}{{n-1}\choose {m-i-j}}\frac{(\sin r)^{n-m+i+j}}{n-m+i+j}(\sin(r+s)')^{m-i-j} \beta_i(m-j, d)\\\label{eq:first}
&+\frac{1}{J_{n,n}(s)}\frac{\mathrm{vol}(S^n)}{2}(\sin r)^n\beta_{m-j}(m-j, d).\end{align}
Setting now $v_n=(1+\frac{\mathrm{vol}(S^n)}{2})$, using \eqref{eq:first} we get
\be\label{eq:inter2} \frac{\mathrm{vol}(\mathcal{T}(A_j, r))}{\mathrm{vol}(B(p, s))}\leq\frac{ v_n}{J_{n,n}(s)}\sum_{i=0}^{m-j}{{n-1}\choose {m-i-j}}\frac{(\sin r)^{n-m+i+j}}{n-m+i+j}(\sin(r+s)')^{m-i-j} \beta_i(m-j, d).\ee
Let us write now:
\be\label{eq:ineq1} {{n-1}\choose {m-i-j}}\frac{1}{n-m+i+j}={{m-j}\choose{i}}\frac{{{n-1}\choose {m-i-j}}\frac{1}{n-m+i+j}}{{{m-j}\choose{i}}} \leq {{m-j}\choose{i}} n^{n-m+j-1}\ee
and, using Lemma \ref{lemma:degree}, let us estimate
\be \label{eq:ineq2}\beta_i(m-j, d)\leq 2(4 d)^{n-m+i+j}.\ee
Using \eqref{eq:ineq1} and \eqref{eq:ineq2} into \eqref{eq:inter2} we get
\begin{align} \frac{\mathrm{vol}(\mathcal{T}(A_j, r))}{\mathrm{vol}(B(p, s))}&\leq\frac{2 v_n}{J_{n,n}(s)}\sum_{i=0}^{m-j} {{m-j}\choose{i}} n^{n-m+j-1}(\sin r)^{n-m+i+j}(\sin(r+s)')^{m-i-j} (4 d)^{n-m+i+j}\\
&=\frac{2 v_n}{J_{n,n}(s)}n^{n-m+j-1}(4 d)^{n-m+j}(\sin r)^{n-m+j}\left(4d\sin r +\sin(r+s)'\right)^{m-j}\\
&\leq \frac{2v_n}{(\sin s)^n}n^{n-m+j}(4 d)^{n-m+j}(\sin r)^{n-m+j}\left(4d\sin r +\sin(r+s)'\right)^{m-j},
\end{align}
where in the last inequality we have used Lemma \ref{lemma:pJ}.

From this we see that
\begin{align}\frac{\mathrm{vol}(\mathcal{T}(A_0, r))}{\mathrm{vol}(B(p, s))}+\frac{\mathrm{vol}(\mathcal{T}(A_1, r))}{\mathrm{vol}(B(p, s))}&\leq \frac{2 v_nn^{n-m}(4 d)^{n-m}(\sin r)^{n-m}}{(\sin s)^n}\left(4d\sin r +\sin(r+s)'\right)^{m-1}\cdot\\
&\quad\cdot (4d\sin r +\sin(r+s)'+n4 d(\sin r))\\
&= \frac{2 v_nn^{n-m}(4 d)^{n-m}(\sin r)^{n-m}}{(\sin s)^n}\left(4d\sin r +\sin(r+s)'\right)^{m-1}\cdot\\
&\quad\cdot ((n+1)4d\sin r +\sin(r+s)')\\
&\leq2 v_n\left(\frac{4n d \sin r}{\sin s}\right)^{n-m}\left((n+1)4d\frac{\sin r}{\sin s} +\frac{\sin(r+s)'}{\sin s}\right)^{m}=(*).
\end{align}
We observe now that for every $0\leq r, s\leq \frac{\pi}{2}$ we have $\sin(r+s)'\leq \sin r+\sin s$ (Lemma \ref{lemma:sin}) and therefore:
\begin{align}
(*)&=2 v_n\left(\frac{n4 d \sin r}{\sin s}\right)^{n-m}\left((n+1)4d\frac{\sin r}{\sin s} +\frac{\sin s+\sin r }{\sin s}\right)^{m}\\
&\label{eq:finalc}\leq  2 v_n\left(\frac{n4 d \sin r}{\sin s}\right)^{n-m}\left(1+(4nd +4d+1)\frac{\sin r}{\sin s}\right)^{m}.
\end{align}
This proves the first part of the statement. 

If now $\sin r\leq\frac{\sin s}{(4nd +4d+1)m}$, then
\be\label{eq:expo}\left(1+(4nd +4d+1)\frac{\sin r}{\sin s}\right)^{m}\leq\left(1+\frac{1}{m}\right)^m\leq e \ee
and the second part follows from \eqref{eq:finalc}.
\end{proof}
It remains to prove the lemma that we used in the proof.
\begin{lemma}\label{lemma:sin}For every $0\leq r, s\leq \frac{\pi}{2}$ we have $\sin\left(\min\{r+s, \frac{\pi}{2}\}\right)\leq \sin r+\sin s.$
\end{lemma}
\begin{proof}If $r+s\leq \frac{\pi}{2}$, then
\be \sin\left(\min\left\{r+s, \frac{\pi}{2}\right\}\right)=\sin (r+s)=\sin r\cos s+\sin s\cos r\leq \sin r+\sin s.\ee
If $r+s\geq \frac{\pi}{2}$, say $r\geq \frac{\pi}{2}-s$, then
\begin{align}\sin\left(\min\left\{r+s, \frac{\pi}{2}\right\}\right) &=\sin\left(\frac{\pi}{2}\right)=1\\
&=(\sin s)^2+(\cos s)^2\underset{0\leq \sin s, \cos s\leq 1}{\leq} \sin s+\cos s\\
&=\sin s+\sin \left(\frac{\pi}{2}-s\right) \leq \sin s+\sin r.
\end{align}
\end{proof}
\subsection{The general spherical case}
We are now ready to give the proof of the bound for the general case in the sphere.
\begin{theorem}\label{thm:gS}Let $\mathcal{P}  \subset \mathbb{R}[X_0,\ldots,X_n]$ be a finite set of homogeneous polynomials of degree bounded by $\delta$ and $Z\subset S^n$ be their common zero set. Assume $\dim_\mathbb{R} (Z)\leq m.$  Given $p\in S^n$ and $\sigma>0$ let $x\in B(p, \sigma)$ be a uniformly distributed point. Then, for every $\varepsilon\geq0$
\be \mathbb{P}(\mathrm{dist}(x, Z)\leq \varepsilon)\leq  2 \left(1+\frac{\mathrm{vol}(S^n)}{2}\right)\left(\frac{8n \delta \sin \varepsilon}{\sin \sigma}\right)^{n-m}\left(1+(8n\delta +8\delta+1)\frac{\sin \varepsilon}{\sin \sigma}\right)^{m}.
\ee
In particular, if $\sin \varepsilon\leq\frac{\sin \sigma}{(8n\delta +8\delta+1)m}$,
\be \mathbb{P}(\mathrm{dist}(x, Z)\leq \varepsilon)\leq  2 e\left(1+\frac{\mathrm{vol}(S^n)}{2}\right) \left(\frac{8n \delta \sin \varepsilon}{\sin \sigma}\right)^{n-m}.\ee
\end{theorem}

\begin{proof}The proof is similar to the proof of Theorem \ref{thm:affine}. Let $0\leq\varepsilon<\frac{\pi}{2}$ (the case $\varepsilon=\frac{\pi}{2}$ follows by a limit argument).

By part (B) of Theorem \ref{thm:approx} there exists a one parameter family of real algebraic sets $\{Z_t\}_{t>0}$ in the sphere $S^n$ and $t_0>0$ such that for all $0<t<t_0$ the set $Z_t$ is a complete intersection defined by homogeneous polynomials $P_1, \ldots, P_{n-m}$ of degree bounded by $2\delta$ and such that:
\be \lim_{t\to0}Z_t\supseteq Z.\ee

Applying Theorem \ref{thm:metric} with the choices $X=B=S^n$ and $\{C_t\}_{t>0}=\{Z_t\}_{t>0}$, for every $0<\tau<\frac{\pi}{2}-\varepsilon$ there exists $t_\tau>0$ such that for all $t<t_\tau$:
\be \mathcal{U}(Z, \varepsilon)\cap B(p, \sigma)\subseteq \mathcal{U}(Z_t, \varepsilon+\tau)\cap B(p, \sigma).\ee
Therefore for every $t<\min\{t_0, t_\tau\}$ we can apply Theorem \ref{thm:completeS} to the set $Z_t$, with the choices $d=2\delta$, $r=\varepsilon+\tau<\frac{\pi}{2}$ and $s=\sigma$:
\be \mathbb{P}(\mathrm{dist}(x, Z)\leq\varepsilon)\leq 2 \left(1+\frac{\mathrm{vol}(S^n)}{2}\right)\left(\frac{8n \delta \sin (\varepsilon+\tau)}{\sin \sigma}\right)^{n-m}\left(1+(8n\delta +8\delta+1)\frac{\sin (\varepsilon+\tau)}{\sin \sigma}\right)^{m}.\ee
Letting $\tau\to 0$ on both sides we get the first part of the result; the second part follows now arguing as in \eqref{eq:expo}.
\end{proof}

\subsection{Proof of Theorem \ref{thm:ill}}\label{sec:pill}
\begin{proof}Using the fact that $\Sigma\subseteq V$ we see that 
\be \{\mathscr{C}(x)\geq t\}\subseteq \left\{\mathrm{dist}(x, Z)\leq \frac{1}{\arcsin t}\right\}.\ee
For $a\in S^n$ and $0<u\leq 1$, let us write $B_{\mathrm{sin}}(a, u)=B(a, \arcsin u)\cup B(-a, \arcsin u)$, so that:
\begin{align} \underset{x\in B_{\sin}(a, u)}{\mathbb{P}}\left\{\mathscr{C}(x)\geq t\right\}&=\frac{\mathrm{vol}\left(\left\{\mathscr{C}(x)\geq t\right\}\cap B_{\mathrm{sin}}(a, u) \right)}{\mathrm{vol}\left(B_{\mathrm{sin}}(a, u)\right)}\\
&\label{eq:summand}=\frac{\mathrm{vol}\left(\left\{\mathscr{C}(x)\geq t\right\}\cap B(a, \arcsin u) \right)+ \mathrm{vol}\left(\left\{\mathscr{C}(x)\geq t\right\}\cap B(-a, \arcsin u) \right)}{2\mathrm{vol}\left(B(a, \arcsin u)\right)}\\
&\leq \frac{\mathrm{vol}\left(\left\{\mathrm{dist}(x, Z)\leq \frac{1}{\arcsin t}\right\}\cap B(a, \arcsin u)\right)}{\mathrm{vol}\left(B(a, \arcsin u)\right)}.
\end{align}
In the last step we have used the fact that $\Sigma=-\Sigma$, which in particular implies that the two summands in the numerator of \eqref{eq:summand} are equal.
The result is now just a reformulation of Theorem \ref{thm:gS}.
\end{proof}

\bibliographystyle{amsplain}
\bibliography{master}

\def\cprime{$'$} \def\cprime{$'$}
\providecommand{\bysame}{\leavevmode\hbox to3em{\hrulefill}\thinspace}
\providecommand{\MR}{\relax\ifhmode\unskip\space\fi MR }
\providecommand{\MRhref}[2]{%
  \href{http://www.ams.org/mathscinet-getitem?mr=#1}{#2}
}
\providecommand{\href}[2]{#2}
\begin{thebibliography}{10}

\bibitem{Barone-Basu}
Sal Barone and Saugata Basu, \emph{Refined bounds on the number of connected
  components of sign conditions on a variety}, Discrete Comput. Geom.
  \textbf{47} (2012), no.~3, 577--597. \MR{2891249}

\bibitem{BPRbook2}
S.~Basu, R.~Pollack, and M.-F. Roy, \emph{Algorithms in real algebraic
  geometry}, Algorithms and Computation in Mathematics, vol.~10,
  Springer-Verlag, Berlin, 2006 (second edition). \MR{1998147 (2004g:14064)}

\bibitem{BCSS}
Lenore Blum, Felipe Cucker, Michael Shub, and Steve Smale, \emph{Complexity and
  real computation}, Springer-Verlag, New York, 1998, With a foreword by
  Richard M. Karp. \MR{1479636}

\bibitem{BottTu}
Raoul Bott and Loring~W. Tu, \emph{Differential forms in algebraic topology},
  Graduate Texts in Mathematics, vol.~82, Springer-Verlag, New York-Berlin,
  1982. \MR{658304}

\bibitem{BuCu}
Peter B\"{u}rgisser and Felipe Cucker, \emph{Condition}, Grundlehren der
  Mathematischen Wissenschaften [Fundamental Principles of Mathematical
  Sciences], vol. 349, Springer, Heidelberg, 2013, The geometry of numerical
  algorithms. \MR{3098452}

\bibitem{BCL2}
Peter B\"{u}rgisser, Felipe Cucker, and Martin Lotz, \emph{Smoothed analysis of
  complex conic condition numbers}, J. Math. Pures Appl. (9) \textbf{86}
  (2006), no.~4, 293--309. \MR{2257845}

\bibitem{BCL}
\bysame, \emph{The probability that a slightly perturbed numerical analysis
  problem is difficult}, Math. Comp. \textbf{77} (2008), no.~263, 1559--1583.
  \MR{2398780}

\bibitem{PSC}
Peter B\"{u}rgisser and Antonio Lerario, \emph{Probabilistic {S}chubert
  calculus}, J. Reine Angew. Math. \textbf{760} (2020), 1--58. \MR{4069883}

\bibitem{Demmel}
James~W. Demmel, \emph{The probability that a numerical analysis problem is
  difficult}, Math. Comp. \textbf{50} (1988), no.~182, 449--480. \MR{929546}

\bibitem{Edelman}
Alan~Stuart Edelman, \emph{Eigenvalues and condition numbers of random
  matrices}, ProQuest LLC, Ann Arbor, MI, 1989, Thesis (Ph.D.)--Massachusetts
  Institute of Technology. \MR{2941174}

\bibitem{Gray1}
Alfred Gray, \emph{Volumes of tubes about complex submanifolds of complex
  projective space}, Trans. Amer. Math. Soc. \textbf{291} (1985), no.~2,
  437--449. \MR{800247}

\bibitem{Gray2}
\bysame, \emph{Tubes}, second ed., Progress in Mathematics, vol. 221,
  Birkh\"{a}user Verlag, Basel, 2004, With a preface by Vicente Miquel.
  \MR{2024928}

\bibitem{Howard}
R.~Howard, \emph{The kinematic formula in {R}iemannian homogeneous spaces},
  Mem. Amer. Math. Soc. \textbf{106} (1993), no.~509, vi+69. \MR{1169230}

\bibitem{Kouchnirenko}
A.~G. Ku{\v{s}}nirenko, \emph{Newton polyhedra and {B}ezout's theorem},
  Funkcional. Anal. i Prilo\v zen. \textbf{10} (1976), no.~3, 82--83.
  \MR{0422272 (54 \#10263)}

\bibitem{Loeser}
Fran\c{c}ois Loeser, \emph{Volume de tubes autour de singularit\'{e}s}, Duke
  Math. J. \textbf{53} (1986), no.~2, 443--455. \MR{850545}

\bibitem{Lotz}
Martin Lotz, \emph{On the volume of tubular neighborhoods of real algebraic
  varieties}, Proc. Amer. Math. Soc. \textbf{143} (2015), no.~5, 1875--1889.
  \MR{3314098}

\bibitem{Srivastasa}
S.~M. Srivastava, \emph{A course on {B}orel sets}, Graduate Texts in
  Mathematics, vol. 180, Springer-Verlag, New York, 1998. \MR{1619545}

\bibitem{Vit2}
A.~G. Vitushkin, \emph{The relation of variations of a set to the metric
  properties of its complement}, Dokl. Akad. Nauk SSSR (N.S.) \textbf{114}
  (1957), 686--689. \MR{0090622}

\bibitem{Vit1}
A.~G. Vitu\v{s}kin, \emph{O mnogomernyh variaciyah}, Gosudarstv. Izdat.
  Tehn.-Teor. Lit., Moscow, 1955. \MR{0075267}

\bibitem{Weyl}
Hermann Weyl, \emph{On the {V}olume of {T}ubes}, Amer. J. Math. \textbf{61}
  (1939), no.~2, 461--472. \MR{1507388}

\bibitem{Wongkew}
Richard Wongkew, \emph{Volumes of tubular neighbourhoods of real algebraic
  varieties}, Pacific J. Math. \textbf{159} (1993), no.~1, 177--184.
  \MR{1211391}

\bibitem{ComteYomdin}
Yosef Yomdin and Georges Comte, \emph{Tame geometry with application in smooth
  analysis}, Lecture Notes in Mathematics, vol. 1834, Springer-Verlag, Berlin,
  2004. \MR{2041428}

\end{thebibliography}

\end{document}